\documentclass[11pt]{amsart} 
\usepackage{amssymb}

\usepackage[T1]{fontenc} 
\usepackage[utf8]{inputenc} 

\usepackage[unicode,colorlinks,
	citecolor=blue,
	linkcolor=blue,
	urlcolor=black,
]{hyperref} 

\usepackage{enumitem} 

\usepackage{xcolor,soul}

\definecolor{Dark Orchid}{RGB}{153,50,204}
\definecolor{Goblin Green}{RGB}{99,181,33}
\definecolor{Crayola Bittersweet}{RGB}{254,111,94}

\theoremstyle{plain}
\newtheorem{thm}{Theorem}[section]

\newtheorem{lem}[thm]{Lemma}
\newtheorem{cor}[thm]{Corollary}

\theoremstyle{definition}
\newtheorem{defn}[thm]{Definition}

\newtheorem{query}[thm]{Question}

\theoremstyle{remark}

\newtheorem*{claim}{Claim}

\numberwithin{equation}{section}

\theoremstyle{remark}
\newcommand{\thistheoremname}{}
\newtheorem*{genremark*}{\thistheoremname}
\newenvironment{case}[1]
	{\renewcommand{\thistheoremname}{#1}%
	\begin{genremark*}}
 	{\end{genremark*}}

\DeclareMathOperator{\uh}{\upharpoonright}

\DeclareMathOperator{\lowww}{low_3}
\DeclareMathOperator{\highh}{high_2}
 \def \angles#1#2{\langle #1,#2 \rangle}
 
\newcommand{\N}{\mathbb{N}}

\begin{document}

\title{Contrasting the Halves of an Ahmad Pair}

\author{Karthik Ravishankar}
\address{Department of Mathematics, University of Wisconsin--Madison, 480 Lincoln Dr., Madison, WI 53706, USA}
\email{kravishanka3@wisc.edu}
\urladdr{https://karthik-ravishankar-math.github.io}

\date{\today}

\makeatletter
\@namedef{subjclassname@2020}{\textup{2020} Mathematics Subject Classification}
\makeatother


\begin{abstract}
We study Ahmad pairs in the $\Sigma^0_2$ enumeration degrees. $(A,B)$ is an Ahmad pair if $A \not \leq_e B$ and every $Z <_e A$ satisfies $Z \leq_e B$. We characterize the degrees that are the left halves of an Ahmad pair as those that are $\lowww$ and join irreducible. We then show that the right half has to be $\highh$ giving a natural separation between the two halves which is a significant strengthening of previous work. We define a hierarchy of join irreducibility notions using which we characterize the left halves of Ahmad $n$-pairs as those that are $\lowww$ and $n$-join irreducible, while the right halves are $\highh$. This allows us to extend and clarify previous work to show that for any $n$, there is a set $A$ which is the left half of an Ahmad $n$-pair but not of an Ahmad $(n+1)$-pair. These results have new implications about the $\forall \exists$-theory of the $\Sigma^0_2$ e-degrees as a partial order and also provide a new $\Pi_3$ definition of $\lowww$ as well as $\highh$.
\end{abstract}

\maketitle

\section{Introduction}
In this paper we work in the local structure of the enumeration degrees, those degrees below $0'_e$. These are precisely the degrees of $\Sigma^0_2$ sets. The degree structure here differs from that of the c.e.\ Turing degrees. One of the differences is the existence of a join irreducible $\Sigma^0_2$ e-degree while the c.e.\ degrees are all join reducible by the Sacks splitting theorem. Downwards density of the enumeration degrees gives an elementary difference from the $\Delta^0_2$ Turing degrees which has minimal elements.

Since the c.e.\ Turing degrees embed into the $\Sigma^0_2$ e-degrees, and any partial order embeds into the c.e.\ Turing degrees, the same holds for the $\Sigma^0_2$ e-degrees. Hence the $\exists$-fragment of the theory of the $\Sigma^0_2$ e-degrees is decidable. Kent \cite{undecid} showed that the $\exists \forall \exists$-fragment is not. Recently there have been renewed attempts at solving the question of decidability of the $\forall \exists$ theory of the $\Sigma^0_2$ e-degrees. The extension of embeddings problem asks, given partial orders $\mathcal P \subseteq \mathcal Q$, if every embedding of $\mathcal P$ can be extended to an embedding of $\mathcal Q$. This is a subproblem of the $\forall \exists$ theory and has been solved by Lempp, Slaman and Sorbi \cite{extemb}. One of the main obstruction towards extending an embedding is the phenomenon of Ahmad pairs. 
\begin{defn}
Two $\Sigma^0_2$ sets $(A,B)$ form \emph{an Ahmad pair} if $A \not \leq_e B$ and  $\forall Z <_e A\; (Z \leq_e B)$. 
\end{defn}

Note that if $\mathcal P$ is a two element antichain embedded as an Ahmad pair, and $\mathcal Q$ extends $\mathcal P$ by containing an element below the left half but not below the right half, then this embedding of $\mathcal P$ cannot be extended to an embedding of $\mathcal Q$.

Ahmad \cite{ahmad} showed that such pairs exist and that if $(A,B)$ is an Ahmad pair then $(B,A)$ cannot be an Ahmad pair, i.e.\ there is no `symmetric' Ahmad pair. Kallimullian, Iskander, Lempp, Ng and Yamaleev \cite{cupping} showed that if $(A,B)$ is an Ahmad pair, then $A\oplus B <_e 0'_e$.

Extending Ahmad's result in a strong way, Goh, Lempp, Ng and Soskova \cite{extensions} showed that there is no Ahmad triple: sets $A,B,C \leq_e 0'_e$ such that $(A,B)$ and $(B,C)$ are both Ahmad pairs. They show this using a direct construction which ends up being a complicated $0'''$ priority argument. 

In this paper we prove three main results. The first one completely characterizes the left half of an Ahmad pair.
\begin{thm} The following are equivalent for a $\Sigma^0_2$ set $A$:
\begin{enumerate}
    \item $A$ is $\lowww$ and join irreducible,
    \item There is a $\Sigma^0_2$ set $B$ such that $(A,B)$ is an Ahmad pair.
\end{enumerate}
\end{thm}
Using this, we get a $\Pi_3$ definition of $\lowww$ in the language of $\{\leq_e\}$ in the local structure in Theorem~\ref{thm:lowdefn}. Ganchev and Soskova \cite{defJumpClass} show that all the jump classes are definable in the local structure. However the definitions go through definability of the corresponding jump classes in the Turing degrees and as such the definitions have high quantifier complexity. 

The second result shows that the right half is fundamentally different from the left half.
\begin{thm}
   Suppose $(A,B)$ is an Ahmad pair where $A, B$ are $\Sigma^0_2$. Then $B$ is $\highh$.
\end{thm}
As a corollary we recover the non-existence of Ahmad triples in a natural way.
\begin{cor}
    There is no Ahmad triple in the local structure, i.e.\ if $A,B$ are $\Sigma^0_2$ sets such that $(A,B)$ is Ahmad pair, then for any $\Sigma^0_2$ set $C$, the pair $(B,C)$ is not an Ahmad pair.
\end{cor}

Moreover, we use our analysis of right halves to get a $\Pi_3$ definition of $\highh$ in Theorem~\ref{thm:defhigh2}.

The authors in \cite{extensions} construct what they call `a weak Ahmad triple': sets $(A,B_0,B_1)$ such that $A\not \leq_e B_0, B_1$ and $\forall Z<_e A\; (Z \leq_e B_0 \text{ or } Z \leq_e B_1)$. We will call this an Ahmad $2$-pair. They construct $A,B_0,B_1$ such that $A$ cannot be the left half of an Ahmad pair while $(A,B_0,B_1)$ is an Ahmad $2$-pair. We generalize this as follows:
\begin{defn}
    $(A, B_0,\dots,B_{n-1})$ is an \emph{Ahmad $n$-pair} if $A \not \leq_e B_i$ for every $i$ and $\forall Z <_e A \exists i< n\; (Z \leq_e B_i)$.
\end{defn}
In Definition~\ref{def:njoinirr}, we define a notion of $n$-join irreducibility. Using this we characterize the left half of an Ahmad $n$-pair.
\begin{thm}
The following are equivalent for a $\Sigma^0_2$ set $A$:
\begin{enumerate}
    \item $A$ is $\lowww$ and $n$-join irreducible,
    \item There are $\Sigma^0_2$ sets $B_0,B_1,\dots, B_{n-1}$ such that $(A,B_0,\dots, B_{n-1})$ form an Ahmad $n$-pair.
\end{enumerate}
\end{thm}
We define the notion of a proper Ahmad $n$-pair in Definition~\ref{def:properpair}. Using this we get the following stronger result.
\begin{thm}
If $(A, B_0,\dots,B_{n-1})$ is a proper Ahmad $n$-pair, then $B_i$ is $\highh$ for every $i<n$.
\end{thm}
This then shows that if $B_i$ is a right half of a proper Ahmad $n$-pair, it cannot be the left half of an Ahmad $m$-pair for any $n,m\geq 1$. This extends a result in an upcoming paper \cite{onePtExt} where they show that at least one of the $B_i$'s is not the left half of an Ahmad $m$-pair for any $m\geq 1$. In Theorem~\ref{thm:nIrrConst}, for every $n\geq 1$, we build a set that is $n$-join reducible and $(n+1)$-join irreducible. This allows us to extend the work in \cite{extensions}. 
\begin{thm}
    For every $n \geq 1$ there is a $\Sigma^0_2$ set $A$ which is the left half of an Ahmad $(n+1)$-pair but not the left half of any Ahmad $n$-pair.
\end{thm}
In Theorem~\ref{thm:AE}, using the results in this paper, we decide a larger fragment of the $\forall \exists$-theory of the local structure.
\section{Preliminaries}
Enumeration reducibility seeks to capture the notion of computational strength, where we only care about `positive' information, i.e.\ if an element is in the set. 

An enumeration of a set is a surjective function from $\omega$ onto the set. Given non-empty sets $A, B \subseteq \omega$ we say $A$ is enumeration reducible to $B$, denoted as $A \leq_e B$, if every enumeration of $B$ computes an enumeration of $A$. By Selman's theorem, this can equivalently be defined as $A \leq_e B$ if there is a c.e.\ set $\Gamma$ such that $x \in A \iff \langle x,e\rangle \in \Gamma$ and $D_e \subseteq B$ (where $D_e$ is a finite set given in terms of its canonical index). Let $\{\Gamma_e\}_e$ be a standard listing of all c.e.\ sets.
\begin{defn}
   For a set $X$ let $K_X = \bigoplus_{e\in \omega} \Gamma_e(X)$. Then the \emph{jump} of $X$ is defined as $X' = K_X \oplus \overline{K_X}$ and the \emph{skip} of $X$ is defined as $X^\diamond = \overline{K_X}$. 
\end{defn}
Note that $A \leq_e X$ if and only if $A \leq_1 K_X$. We will let $X^{\langle n \rangle}$ denote the $n^{th}$ skip of $X$. In this paper we use $X'$ to mean the enumeration jump, not the Turing jump. To avoid confusion, in the case of $\emptyset$, we denote its jump by $0'_e$. The collection $\{X \subseteq \omega: X \leq_e 0'_e\}$ are precisely the $\Sigma^0_2$ sets. 

We say a set $X$ is total if $\overline X \leq_e X$, i.e.\ the positive part of the set can tell us about its negative part. The map $\iota : \mathcal P(\omega) \to \mathcal P(\omega)$ defined by $\iota(X) = X \oplus \overline X$ induces an embedding of the Turing degrees into the enumeration degrees as a partial order preserving join and jump. The image of this embedding is precisely the total degrees (degrees containing a total set).

The authors in \cite{cototSkip} show that the jump and skip operation coincide precisely on the cototal degrees---degrees containing a set whose complement is total. The cototal degrees include the total degrees as well as the $\Sigma^0_2$ degrees. In particular, if $X\leq_e 0'_e$ then $X' \equiv_e X^\diamond$, however they are not $1$-equivalent in general.

Note that $X' >_e X$ for every $X \subseteq \omega$. The authors in \cite{cototSkip} construct sets $Y\subseteq \omega$ for which $Y^\diamond \not \geq_e Y$. Therefore the jump and skip are not the same operations on the global structure of the enumeration degrees.

Recall that a $\Sigma^0_2$ set $X$ is \emph{$\lowww$} if $X''' \leq_e 0'''_e$ and is \emph{$\highh$} if $X'' \geq_e 0'''_e$. The following characterization will be useful.
\begin{lem}\label{lem:poststheorem}
The following are equivalent for a set $X$ and for $n \in \omega$.
\begin{enumerate}
\item $X \leq_e 0^{(n)}$.
\item $X \leq_e 0^{\langle n \rangle}$.
\item $X$ is $\Sigma^0_{n+1}$.
\end{enumerate}
\end{lem}

If $A, B \leq_e C$, then so is their union, intersection and projection. This will prove useful to analyze the complexity of many sets in the rest of this paper.
\begin{lem} \label{lem:andorproj}
Let $A = \Gamma_m(C)$ and $B = \Gamma_n(C)$.
\begin{enumerate}
\item $A \cup B = \Theta(C)$ where $\Theta = \{\langle x , D \rangle:  \langle x, D\rangle \in \Gamma_m \text{ or } \langle x, D\rangle \in \Gamma_n\}$. 
\item $A \cap B = \Delta(C)$ where $\Delta = \{\langle x , D \rangle: \exists \langle x, E\rangle \in \Gamma_m, \langle x, F\rangle \in \Gamma_n (D = E\cup F)\}$.
\item $\{x : \exists i \; \langle x, i \rangle \in A\} = \Phi(C)$ where $\Phi = \{\langle x,D\rangle : \exists i\; \langle \langle x,i\rangle,D\rangle \in \Gamma_m\}$.
\end{enumerate}
\end{lem}

In the rest of this paper, unless mentioned otherwise, all the sets will be $\Sigma^0_2$ sets. For the few results which apply globally, we shall explicitly mention the more general fact. All other results only apply to the local structure.

\section{Ahmad Sequences and the Left Half}
For any set $A$, we can always uniformly enumerate $\Sigma^0_2$ indices for all the sets in $\{X\subseteq \omega: X\leq_e A\}$ by the parameter theorem. However to build Ahmad pairs we need uniform access to the sets \emph{strictly} below $A$. The notion of an Ahmad sequence captures this precisely.
\begin{defn}
A computable sequence of $\Sigma^0_2$ sets $\{\Gamma_{f(n)}(0'_e)\}_n$ is an \emph{Ahmad sequence} for $A$ if $\{\Gamma_{f(n)}(0'_e)\}_n = \{ X \subseteq \omega: X <_e A\}$. 
\end{defn}
We will show that the left half of an Ahmad pair always has an Ahmad sequence. To this end, we first show that given two uniform sequences of sets, we can uniformly enumerate their intersection.
\begin{lem}\label{lem:uniformEnum}
    Let $f, g$  be computable functions with $\{ F \subseteq \omega: |F| < \infty\} \subseteq \{\Gamma_{f(n)}(0'_e)\}_n, \{\Gamma_{g(n)}(0'_e)\}_n$. Then there is a computable function $h$ such that $\{\Gamma_{h(n)}(0'_e)\}_n = \{\Gamma_{f(n)}(0'_e)\}_n \cap \{\Gamma_{g(n)}(0'_e)\}_n$.
\end{lem}
\begin{proof}
Recall that the $\Sigma^0_2$ sets are precisely those which are c.e.\ relative to $0'$ and we can effectively transform between these two indices. So without loss of generality, let $\{W^{0'}_{f(n)}\}_n$ and $\{W^{0'}_{g(n)}\}$ be the uniform families of $\Sigma^0_2$ sets.

Then we define $h(n)$ as follows:
\[
x\in W_{h(\langle e,i\rangle),s}^{0'}  \iff x\in W_{f(e),s}^{0'}\cap W_{g(i),s}^{0'} \text{ and } W_{f(e),s}^{0'}\uh x = W_{g(i),s}^{0'}\uh x. 
\]

If $W_{f(e)}^{0'}= W_{g(i)}^{0'}$ then we end up successfully copying $W_{f(e)}^{0'}$ into $W_{h(\langle e,i\rangle)}^{0'}$. On the other hand if they disagree, say $x \in W_{f(e)}^{0'}$ and $x \not \in W_{g(e)}^{0'}$, then let $s$ be the least such that $x \in W_{f(e),s}^{0'}$. For all stages $t>s$ we have $W_{f(e),t}^{0'}\uh x \neq W_{g(i),t}^{0'}\uh x$ and so $W_{h(\langle e,i\rangle)}^{0'} \subseteq W_{h(\langle e,i\rangle),s} \cup \{0,1,\dots,x\}$ and is finite.
\end{proof}
\begin{thm}\label{thm:PairSeq}
Let $(A,B)$ be an Ahmad pair. Then $A$ has an Ahmad sequence.
\end{thm}
\begin{proof}
By the Ahmad pair property, the sets strictly below $A$ are precisely the sets $\{Z: Z\leq_e A\}\cap \{Z:Z\leq_e B\}$. By the parameter theorem both of these families are uniform families of sets and so by the lemma above, $A$ has an Ahmad sequence relative to $C$.
\end{proof}
We can generalize this to the weaker notion of Ahmad $n$-pairs.

\begin{thm}\label{thm:NPairSeq}
Let $(A, B_0,\dots,B_{n-1})$ be an Ahmad $n$-pair. Then $A$ has an Ahmad sequence.
\end{thm}
\begin{proof}
We can uniformly list the sets $\mathcal F_i := \{Z :Z \leq_e A,B_i\}$. Let $\mathcal F = \bigcup_{i<n} \mathcal F_i$. Then $\mathcal F$ has a uniform enumeration which will be the required Ahmad sequence of $A$. 
\end{proof}
Not all degrees have an Ahmad sequence, for an easy example we have the following:
\begin{thm}\label{thm:simpleExample}
$0'_e$ does not have an Ahmad sequence. 
\end{thm}
\begin{proof}
The set $\text{Com}=\{e : W_e \equiv_T 0'\}$ of complete c.e.\ sets is $\Sigma^0_4$ complete \cite{soare}.
Suppose $0'_e$ had an Ahmad sequence $\{\Gamma_{f(n)}(0'_e)\}_n$. We will use this to flip quantifiers. Note that $W_e$ is complete if and only if $W_e\oplus\overline W_e \equiv_e \overline W_e \neq \Gamma_{f(n)}(0'_e)$ for every $n$. Let $A_{n,s}$ be a $\Sigma^0_2$ approximation to $\Gamma_{f(n)}(0'_e)$. So $e\in \text{Com} $ if and only if for every $n$ there is an $m,s$ with
\begin{align*}
    \forall t>s&(m \in W_{e,s}\text{ and } \exists r>t\; (m \not \in A_{n,r})) \text{ or }
    (m  \not \in W_{e,t} \text{ and } m \in A_{n,t}). 
\end{align*}
But this is a $\Pi^0_4$ formula, contradicting the $\Sigma^0_4$ completeness of Com. 
\end{proof}
The rest of this section is devoted to completely characterizing the $\Sigma^0_2$ sets with an Ahmad sequence. We write $X^{[i]}$ for the $i^{th}$ column of $X$.

\begin{lem}\label{lem:pi3complete}
    There is a computable function $g$ such that: 
    \begin{enumerate}
        \item $e \in A^{\langle 3 \rangle} \iff \Gamma_{g(e)}^{[i]}(A)$ is finite for every $i$.
        \item $e \not \in A^{\langle 3 \rangle} \iff \Gamma_{g(e)}^{[i]}(A) = \omega$ for some $i$ and finite otherwise.
    \end{enumerate}
\end{lem}
\begin{proof}
 $A^{\langle 3 \rangle}$ is $\Pi_3$ asking only negative information about $A$:
 \begin{align*}
 e \in A^{\langle 3 \rangle} &\iff e \not \in \Gamma_e(A^{\langle 2 \rangle}) \iff \forall \langle e, D \rangle \in \Gamma_e (D \not \subseteq A^{\langle 2 \rangle})
 \\
 &\iff \forall \langle e,D \rangle \in \Gamma_e \exists n \in D (n \in \Gamma_n(\overline {K_A})) \\ &\iff \forall \langle e,D \rangle \in \Gamma_e \exists  \langle n , F \rangle \in \Gamma_n \; (n \in D \;\text{ and }\; F \subseteq \overline{K_A}).
 \end{align*}
  To simplify notation let $e\in A^{\langle 3 \rangle} \iff \forall i \exists j  R(e,i,j)$ where $\neg R(e,i,j) \leq_e A$. Then define $g$ as follows:  Given an $e$, let $X_e$ be the set where we enumerate $j$ into $X_e^{[i]}$ if $\forall j'<j$ $\neg R(e,i,j')$. Then $X_e \leq_e A$ and since we can go from $e \to X_e$ uniformly, there is a computable $g$ such that $X_e = \Gamma_{g(e)}(A)$.
  
 Now if $e \in A^{\langle 3 \rangle}$, then for every $i$ we have $\Gamma_{g(e)}^{[i]}(A)\subseteq \omega \uh j$ for the least $j$ for which $ R(e,i,j)$ holds.  If $e \not \in A^{\langle 3 \rangle}$, let $i$ be the first such that $\forall j  \neg R(e,i,j)$. Then $\Gamma_{g(e)}^{[i]} = \omega$.
\end{proof}
\begin{defn}
Let $X,Y$ be $\Sigma_2^0$ sets.
\begin{enumerate}
\item A \emph{good approximation} to $X$ is a computable sequence $\{X_s\}_s$ of finite sets with infinitely many \emph{good stages} $G_X :=\{s: X_s \subseteq X\}$ such that $\lim_{s \in G_X} X_s(n) = X(n)$ for every $n$.
\item A \emph{correct approximation} $\{Y_s\}$ to $Y$ with respect to a good approximation $\{X_s\}$ to $X$ is an approximation where $G_X \subseteq G_Y$ and $\lim_{s \in G_X} Y_s(n) = Y(n)$.
\end{enumerate}
\end{defn}
Every $\Sigma_2^0$ set $X$ has a good approximation (implict in \cite{cooper})  and if $Y = \Gamma_e(X)$, the natural approximation $\{\Gamma_{e,s}(X_s)\}$ to $Y$ is a correct (but not necessarily good) approximation with respect to a good approximation $\{X_s\}$ to $X$. Also observe that if $Y$ has a correct approximation with respect to a good approximation of $X$, then $Y \leq_e  X$.

The motivation behind the notion of correct approximations is that if we restrict these approximations to good stages, they behave like $\Delta^0_2$ approximations. In particular, they have a limit. We then get the desirable property that the length of agreement goes to $\infty$ precisely when the two sets are equal.
\begin{lem}\label{lem:incompleteset}
    Let $A$ be a non-c.e.\ set with a good approximation $\{A_s\}_s$ to $A$. There is a uniform sequence $\{ \Theta_e\}$ of enumeration operators (depending on an approximation to $A$) such that: 
    \begin{itemize}
        \item $\Theta_e(A) <_e A \iff \Gamma_e^{[i]}(A)$ is finite for every $i$ and 
        \item $\Theta_e(A) \geq_e A \iff \Gamma_e^{[i]}(A)$ is infinite for some $i$.
    \end{itemize}
    Moreover, given any $\Delta(A) <_e A$, we can ensure that $\Theta_e(A) \geq_e \Delta(A)$ for every $e$.
\end{lem}
\begin{proof}
   Let $X = \Gamma_e(A)$ and let $\{A_s\}_s$ be a good approximation to $A$. We shall build the enumeration operator $\Theta_e$ to meet the requirements:
\begin{align*}
   \mathcal R_i: \Gamma_i(\Theta_e(A)) \neq A \text{ iff } X^{[\leq i]} \;\text{is finite}.
\end{align*}
 
At stage $s=0$, let $\Theta_e$ contain the axioms $\angles{\angles{0}{x}}{F}$ for every axiom $\angles{x}{F} \in \Delta$.

At stage $s+1$, we have substages $t\leq s$. At substage $t$ we do the following:
\begin{enumerate}
    \item Let $l_{t,s} = l(\Gamma_{t,s}(\Theta_e(A_s)),A_s) = $ the greatest $n$ such that for every $m<n$ we have $\Gamma_{t,s}(\Theta_{e,s}(A_s))(m) = A_s(m)$. 
    
    For every $m < l_{t,s}$ with $m \in \Gamma_{t,s}(\Theta_{e,s}(A_s)^{[\leq t]}\cup \mathbb N^{[>t]})$, let $\langle m, D\rangle \in \Gamma_{t,s}$ be the first axiom witnessing this. For every $y \in D^{[> t]}$, add the axiom $\langle y, A_s\rangle$ into $\Theta_e$.
    \item Copy $A\uh n$ where $n = |X^{[t]}_s|$ into the $t^{th}$ column by enumerating axioms $\langle \langle t,x\rangle ,\{x\}\cup A_s\rangle$ into $\Theta_e$ for every $x \leq n$. 
\end{enumerate}
This ends the construction.
\begin{claim} If $X^{[\leq i]}$ is finite then $\mathcal R_i$ enumerates finitely many elements into $\Theta_e(A)$ and satisfies its requirement.
\end{claim}
\begin{proof}[Proof of Claim]
   Note that only actions taken at good stages $G_A$ in the approximation $\{A_s\}$ to $A$ will affect $\Theta_e(A)$. We will prove the claim by induction on $i$. Suppose the claim holds for $j < i$. Recall that $\lim_{s \in G_A} l_{i,s}\to \infty \iff \Gamma_i(\Theta_e(A)) = A$. 
   
   So suppose $\Gamma_i(\Theta_e(A))=A$ and $l_{i,s}$ is unbounded. Then  $\Gamma_i(\Theta_e(A)) = \Gamma_i(\Theta_e(A)^{[\leq i]} \cup \mathbb N^{[>i]})$ by the action taken in step $(\romannumeral 1)$. Since each $\mathcal R_j$ enumerates finitely many elements for $j<i$ and since $|X^{[\leq i]}|$ is finite, we have $\Theta_e(A)^{[\leq i]}\leq_e \Delta(A)$ and so $A \leq_e \Delta(A)$. This contradicts the hypothesis and so $\Gamma_i(\Theta_e(A)) \neq A$ and $\mathcal R_i$ only enumerates finitely many elements into $\Theta_e(A)$. 
\end{proof}
The claim above shows that if $X^{[i]}$ is finite for every $i$, then $\Delta(A) \leq_e \Theta_e(A) <_e A$. On the other hand if  $X^{[i]}$ is infinite for some $i$, then $\Theta(A)^{[i]} =^* A$ for the first such $i$ as the encoding in step $(\romannumeral 2)$ is successful. Therefore $\Theta_e(A) \geq_e A$. Note that, given an index for $\Delta$ and an index for the good approximation $\{A_s\}$, we can uniformly in $e$ find an index for $\Theta_e$ .
\end{proof}
\begin{lem}\label{lem:incompbelowskip}
    Let $X_A = \{e: \Gamma_e(A) <_e A\}$. Then if $A$ is non-c.e.\ $X_A \equiv_1 A^{\langle 3 \rangle}$. Fix an enumeration operator $\Delta$ and let $X_A^\Delta =\{e : \Delta(A) \leq_e \Gamma_e(A)<_e A\}$. If $\mathcal F \subseteq \omega$ is any set with $X_A^\Delta \subseteq \mathcal F \subseteq X_A$, then $\mathcal F \geq_1 A^{\langle 3 \rangle}$ whenever $\Delta(A) <_e A$.
\end{lem}
\begin{proof}
    Observe that $\Gamma_e(A)<_e A$ if and only if $\forall i \exists n (\Gamma_i(\Gamma_e(A))(n) \neq A(n))$. We can expand $\Gamma_i(\Gamma_e(A))(n) \neq A(n)$ as: 
    \[
        (n \in A \text{ and } n \not \in \Gamma_i(\Gamma_e(A)) ) \text{ or } (n \not \in A \text{ and } n \in \Gamma_i(\Gamma_e(A))).
    \]
    Each of these terms is $\leq_1 K_A$ or $\leq_1 \overline {K_A}$.  Since $A$ is $\Sigma^0_2$, $A^\diamond \geq_e K_A$, therefore by Lemma~\ref{lem:andorproj}, $\exists n (\Gamma_i(\Gamma_e(A))(n) \neq A(n)) \leq_e A^\diamond$ while its negation is below $A^{\langle 2\rangle}$. 
    
    Applying Lemma~\ref{lem:andorproj} again, $\exists i \neg(\exists n (\Gamma_i(\Gamma_e(A))(n) \neq A(n))) \leq_e A^{\langle 2 \rangle}$, and so is $\leq_1 K_{A^{\langle 2 \rangle}}$. Therefore its negation is $\leq_1\overline{K_{A^{\langle 2 \rangle}}} = A^{\langle 3 \rangle}$ and so $X_A \leq_1 A^{\langle 3 \rangle}$.

    On the other hand, if $\Delta(A)<_e A$, then from Lemma \ref{lem:pi3complete} and Lemma \ref{lem:incompleteset} above, there is a computable function $f$ such that $\Delta(A) \leq_e \Gamma_{f(n)}(A) \leq_e A$ for every $n$ and 
    \begin{align*}
    n \in A^{\langle 3 \rangle} &\iff \Gamma_{f(n)}(A) <_e A,\\
    n \not \in A^{\langle 3 \rangle} &\iff \Gamma_{f(n)}(A) \equiv_e A.
    \end{align*}
    
    Therefore $n \in A^{\langle 3 \rangle} \iff f(n) \in \mathcal F$, so $A^{\langle 3 \rangle} \leq_1 \mathcal F$. 
\end{proof}
\begin{thm}\label{thm:main}
The following are equivalent: 
\begin{enumerate}
\item $A$ has an Ahmad sequence $\{\Gamma_{f(n)}(0'_e)\}_n$.
\item $ X_A = \{e : \Gamma_e(A)<_e A\} $ is $\Sigma^0_4$ (and hence $\Delta^0_4$).
\item$A$ is $\lowww$.
\end{enumerate}
\end{thm}
\begin{proof}
\begin{case}{ $(\romannumeral 1) \implies (\romannumeral 2)$}
    The natural definition of $X_A$ in Lemma~\ref{lem:incompbelowskip} is always $\Pi^0_4$ since $X_A \leq_1 A^{\langle 3 \rangle} \leq_1 0^{\langle 4\rangle}$. We use the Ahmad sequence to get a $\Sigma^0_4$ definition for $X_A$. 
    Note that $e \in X_A$ precisely if $\exists n \forall m (\Gamma_{f(n)}(0'_e)(m) = \Gamma_e(A)(m))$. 
    
    Let $C = 0'_e$. The terms $m \in \Gamma_{f(n)}(0'_e)$ and $m \not \in \Gamma_{f(n)}(0'_e)$ are $\leq_1 C'$ while the terms $m \in \Gamma_{e}(A)$ and $m \not \in \Gamma_{e}(A)$ are $\leq_1 A'$. Since $C$ is $\Sigma^0_2$, $A' \leq_e C' \leq_e C^\diamond$.
    
    Therefore by Lemma~\ref{lem:andorproj}, checking if $\exists m (\Gamma_{f(n)}(0'_e)(m) \neq \Gamma_e(A)(m))$ is enumeration below $C^\diamond$ while its negation is $\leq_1 \overline{K_{C^\diamond}} = C^{\langle 2 \rangle}$. Finally, using Lemma~\ref{lem:andorproj} again, $\exists n \neg \exists m (\Gamma_{f(n)}(0'_e)(m) \neq \Gamma_e(A)(m))$ is $\leq_e C^{\langle 2\rangle} \leq_e 0^{\langle 3 \rangle}_e$. Therefore, by Lemma~\ref{lem:poststheorem}, $X_A$ is $\Sigma^0_4$.

\end{case}

\begin{case}{$(\romannumeral 2) \implies (\romannumeral 3)$}
    By Lemma~\ref{lem:incompbelowskip}, $A^{\langle 3 \rangle} \equiv_1 X_A$. Therefore since $X_A$ is $\Sigma^0_4$, so is $A^{\langle 3 \rangle}$. Then by Lemma~\ref{lem:poststheorem}, $A^{\langle 3 \rangle} \leq_e 0^{\langle 3 \rangle}$. Since $A$ is $\Sigma^0_2$, $A''' \equiv_e A^{\langle 3 \rangle}$, and so $A$ is $\lowww$.
\end{case}
\begin{case}{$(\romannumeral 3) \implies (\romannumeral 1)$}
      Let $C = 0'_e$. Since $A^{\langle 3 \rangle} \leq_e 0^{\langle 3 \rangle}$, we know $A^{\langle 3 \rangle} \equiv_1 X_A \leq_e C^{\langle 2 \rangle}$. Then $X_A \leq_1 K_{C^{\langle 2 \rangle}}$ and $\overline {X_A} \leq_1 \overline {K_{C^{\langle 2 \rangle}}} = C^{\langle 3 \rangle}$. Then, from the analysis in Lemma~\ref{lem:pi3complete}, there is an $R\leq_e C$ such that
     \[
     e \in X_A \iff \exists n \forall m R(e,n,m).
     \]
     
    As in Lemma \ref{lem:pi3complete}, let $h(e,n)$ be a computable function defined by 
    \[ m \in \Gamma_{h(e,n)}(C) \iff \forall m'<m \;R(e,n,m').\]
    
    By the parameter theorem, since $C \geq_e A$, there is a computable function $k$ such that $\Gamma_{k(e)}(C) = \Gamma_e(A)$. Then the sequence $\{ \Gamma_{k(e)}(C) \cap \Gamma_{h(e,n)}(C)\}_{e,n\in \omega}$ is an Ahmad sequence for $A$: $e\in X_A$ iff $\Gamma_{k(e)}(C) \cap \Gamma_{h(e,n)}(C) = \Gamma_e(A)$ for some $n$ and finite otherwise, while $e \not \in X_A$ iff $\Gamma_{k(e)}(C) \cap \Gamma_{h(e,n)}(C)$ is finite for every $n$. \qedhere
\end{case}
\end{proof}

\section{\texorpdfstring{Join Irreducibility and Building Ahmad $n$-pairs}{Join Irreducibility and Building Ahmad pairs}}

Suppose $(A,B)$ is an Ahmad pair. Then $A$ has an Ahmad sequence and is join irreducible. In this section, we will show that the converse holds as well. The following lemma helps us do this by letting us code an ideal with a uniform enumeration while avoiding being enumeration above A. 

\begin{lem}\label{lem:boundIncomp}
Let $f$ be computable and let $\mathcal F = \{\Gamma_{f(n)}(0'_e)\}_n$ be an ideal such that $A \not \in \mathcal F$. Then there is a set $B\leq_e 0'_e$, such that $A\not \leq_e B$ and $\forall X \in \mathcal F (X \leq_e B)$. 
\end{lem}
\begin{proof}
Let $\{A_s\}_s$ and $\{B_{n,s}\}_s$ be correct approximations to $A$ and $\Gamma_{f(n)}(A)$ respectively, with respect to a good approximation $\{K_s\}$ to $0'_e$. We will build an enumeration operator $\Theta$ so that $B = \Theta(0'_e)$ will meet the requirements: 
\begin{align*}
  \mathcal{N}_e&: A \neq \Gamma_e(B), \\
  \mathcal{P}_e&: \Gamma_{f(e)}(0'_e) =^* B^{[e]}.
\end{align*}
At stage $s=0$, let $\Theta = \emptyset$.\\
At stage $s+1$:
\begin{enumerate}
\item 
For $e\leq s$ let  $l_{e,s} = l(A_s, \Gamma_{e,s}(B_s))$. Then for every $x < l_{e,s}$, if $x \in \Gamma_{e,s}(B_s^{[\leq e]} \cup \N^{[>e]})$, pick the least axiom $\langle x, D \rangle \in \Gamma_e$ which witnesses this. Now for every $y \in D^{[> e]}$ enumerate the axiom $\langle y, K_s \rangle$ into $\Theta$. 
\item
For $e\leq s$, and for every axiom $\langle x, D \rangle \in \Gamma_{f(e),s+1}$, enumerate $\langle \langle e,x\rangle, D\rangle$ into $\Theta$.
\end{enumerate}
\textbf{Verification}: Let $G= \{s: K_s \subseteq 0'_e\}$ be the set of good stages. Recall that $A = \Gamma_e(B) \iff \lim_{s\in G} l_{e,s}= \infty$. Observe that by step $(\romannumeral 2)$, we have ensured that $\Gamma_{f(e)}(0'_e) \subseteq B^{[e]}$.
\begin{claim} Every $\mathcal N_e$ strategy satisfies its requirement and only enumerates finitely many elements into $B$.
\end{claim}
\begin{proof}[Proof of Claim]
    We prove this by induction on $e$. Suppose $\mathcal N_i$ enumerates finitely many elements into $B$ for every $i<e$. Observe that $\Gamma_e(B)\subseteq \Gamma_e(B^{[\leq e]} \cup \mathbb N^{[>e]})$. Suppose $\Gamma_e(B) = A$ and so $\{l_{e,s}\}_{s\in G}$ is unbounded. Now if $x \in \Gamma_e(B^{[\leq e]} \cup \mathbb N^{[>e]})$ witnessed by $\langle x, D \rangle \in \Gamma_e$, then $\mathcal N_e$ will dump $D^{[>e]}$ into $B$ and so $x \in \Gamma_e(B)$. So $\Gamma_e(B) = A$ implies $\Gamma_e(B) = \Gamma_e(B^{[\leq e]} \cup \N^{[>e]})$.

    Now for $i \leq e$, $\mathcal P_i$ and $\mathcal N_{j}$ for $j<i$ are the only strategies that enumerates elements into $B^{[i]}$ and $\mathcal P_i$ is successful in coping in $\Gamma_{f(e)}(0'_e)$ into this column up to finite difference by the induction hypothesis.
    
    Now $\{l_{e,s}\}_{s\in G}$ is unbounded $\iff$ $\Gamma_e(B^{[\leq e]} \cup \mathbb N^{[>e]}) = \Gamma_e(B) = A \iff$  $A \leq_e \bigoplus_{i\leq e} B^{[i]} \equiv_e \bigoplus_{i \leq e} \Gamma_{f(i)}(0'_e)$. But since $\mathcal F$ is an ideal and $A \not \in \mathcal F$, this is not possible and so $\{l_{e,s}\}_{s\in G}$ must be bounded. Therefore $\mathcal N_e$ is met and it only enumerates finitely many elements into $B$.
\end{proof}
Since all the $\mathcal N$ requirements enumerate finitely many elements into any column, the $\mathcal P_e$ requirements are successful in ensuring that $B^{[e]} =^* \Gamma_{f(e)}(0'_e)$. \qedhere
\end{proof}
\newpage
\begin{cor}\label{cor:pairiff}
  $A$ is join irreducible and $\lowww$ if and only if there is a set $B$ such that $(A,B)$ form an Ahmad pair.
\end{cor}
\begin{proof}
  As discussed above, if $(A,B)$ is an Ahmad pair, them by Theorem~\ref{thm:PairSeq} $A$ has an Ahmad sequence and so by Theorem~\ref{thm:main} it is $\lowww$. If $X,Y <_e A$, then $X, Y \leq_e B$ and so $X\oplus Y \leq_e B$ i.e.\ $X\oplus Y <_e A$. Therefore $A$ is join irreducible.
  
  For the forward direction, since $A$ is $\lowww$, by Theorem~\ref{thm:main} it has an Ahmad sequence $\mathcal F = \{ \Gamma_{f(n)}(0'_e)\}$ for some computable $f$. Since $A$ is join irreducible, $\mathcal F$ is closed under join and is therefore an ideal. By Lemma~\ref{lem:boundIncomp}, there is a $B \leq_e 0'_e$, such that $(A,B)$ is an Ahmad pair.
\end{proof}
We can now define $\{A\leq_e 0'_e : A \text{ is } \lowww\}$ using a $\Pi_3$ sentences in  the language $\{\leq_e\}$ in the local structure.
\begin{thm}\label{thm:lowdefn}
$A$ is $\lowww$ if and only if $\forall Z \leq_e A(\exists Y <_e Z \implies \exists Y<_e Z, \exists B \not \geq_e Y  (\forall X <_e Y(X\leq_e B)))$.
\end{thm}
\begin{proof}
Kent and Sorbi \cite{kentBounding} show that for every $\Sigma^0_2$ set $X>_e 0_e$, there is a non-c.e.\ $Y \leq_e X$ such that $Y$ is join irreducible. Therefore for every $A$ and $Z\leq_e A$, if $Z$ is non-c.e. then there is a non-c.e.\ $Y \leq_e Z$ such that $Y$ is join irreducible. 

If $A$ is $\lowww$ then $Y$ is $\lowww$ as well. Therefore by Corollary~\ref{cor:pairiff}, there is a $\Sigma^0_2$ set $B$ such that $(Y,B)$ is an Ahmad pair.

On the other hand suppose $A$ is not $\lowww$. Ganchev and Sorbi \cite{initSeg} argue that for every set $C$, there is an initial segment $\mathcal I = \{X: 0<_e X \leq_e D\}$ for some non-c.e.\ $D \leq_e C$, such that for every $X \in \mathcal I$, $X' = C'$. Let $\mathcal I_A$ be such an initial segment below $A$ given by $\mathcal I_A = \{X: 0<_e X \leq_e Z\}$. Then every $Y \in \mathcal I_A$ is not $\lowww$ and so cannot be the left half of an Ahmad pair.
\end{proof}
We extend the characterization in Corollary~\ref{cor:pairiff} to Ahmad $n$-pairs. 
\begin{defn}\label{def:njoinirr}
A set $A$ is \emph{$n$-join irreducible} if $\forall A_0,\dots,A_n <_e A$ there is a pair $i\neq j$ with $A_i \oplus A_j <_e A$
\end{defn}

Note that $1$-join irreducible is the same as join irreducible.
\begin{lem} \label{lem: idealcover}
$A$ is $n$-join irreducible if and only if $\{Z: Z<_e A\}$ is the union of $n$ ideals.
\end{lem}
\begin{proof}
\begin{case}{$(\impliedby)$} Given any $n+1$ degrees strictly below $A$, by the pigeon hole principle, two of them must lie in the same ideal. Then the join of these two degrees is also in that ideal, and so is $<_e A$. Therefore $A$ is $n$-join irreducible.
\end{case}
\begin{case}{$(\implies)$}We prove the claim by induction on $n$. For $n=1$, since join $1$-irreducible is the same as join irreducible, $\{Z: Z<_e A\}$ is an ideal and so covers itself.

Now suppose the claim holds for $n-1 \geq 1$ and let $A$ be $n$-join irreducible. If $A$ is $(n-1)$-join  irreducible, then by the induction hypothesis it has a cover by $n-1$ ideals and we can extend this trivially with the ideal $\{0_e\}$ and we are done. 

So suppose $A$ is $(n-1)$-join reducible and let $X_1,\dots,X_n <_e A $ be such that $X_i \oplus X_j \equiv_e A$ for $i\neq j$.  Let $\mathcal F_i = \{Y: Y\oplus X_i <_e A\}$. We will show that $\{Z : Z <_e A\} = \bigcup_i \mathcal F_i$.

If $Z<_e A$, then $Z$, $X_1, \dots , X_n$ are $n+1$ sets strictly below $A$ and $X_i\oplus X_j \equiv_e A$ for $i\neq j$. Therefore since $A$ is $n$-join irreducible, $Z\oplus X_i <_e A$ for some $i$.

If $Y,Z \in \mathcal F_i$ then consider the $n+1$ sets $Y\oplus X_i, Z\oplus X_i$ along with $X_j$ for $j\neq i$. The only pair here whose join could be strictly below $A$ would be $Y \oplus Z \oplus X_i$. Therefore $Y\oplus Z \in \mathcal F_i$, i.e.\ $\mathcal F_i$ is an ideal. \qedhere
\end{case} 
\end{proof}

\begin{defn}\label{def:properpair}
For an Ahmad $n$-pair $(A,B_1,\dots,B_n)$, we say $B_i$ is \emph{essential} if there is a $C<_e A, B_i$ such that $C \not \leq_e B_j$ for $j\neq i$. We say that an Ahmad $n$-pair $(A,B_1,\dots,B_n)$ is \emph{proper} if every $B_i$ is essential.
\end{defn}
Note that if $B_i$ is essential, we cannot get rid of it to get an Ahmad $(n-1)$-pair with the remaining $B_j$'s for $j\neq i$. A proper Ahmad $n$-pair has no redundant $B_i$'s. Similarly we can define essential for the ideals below $A$.
\begin{defn}
If $\{Z : Z<_e A\} = \bigcup_1^n \mathcal F_i$, where each $\mathcal F_i$ is an ideal, then $\mathcal F_i$ is \emph{essential} if $\mathcal F_i \not \subseteq \mathcal F_j$ for any $j\neq i$. Such a covering is \emph{proper} if each $\mathcal F_i$ is essential.
\end{defn}
\begin{thm}\label{thm:charirr}
The following are equivalent for any set $A\subseteq \omega$:
\begin{enumerate}
\item $A$ has a unique proper covering by $n$ ideals.
\item $A$ has a proper covering by $n$ ideals.
\item $A$ is $n$-join irreducible and $(n-1)$-join reducible.
\end{enumerate} 
\end{thm}
\begin{proof}
$(\romannumeral 1)\implies (\romannumeral 2)$ is trivial. For $(\romannumeral 2)\implies (\romannumeral 3)$, let $\{Z : Z<_e A\} = \bigcup_1^n \mathcal F_i$ be a proper covering. Let $Z_{ij} \in \mathcal F_i-\mathcal F_j $. Let $Z_i := \bigoplus_{j\neq i} Z_{ij}$. Then for every $i$, $Z_i \in \mathcal F_i$, since $\mathcal F_i$ is an ideal. If $Z_i \in \mathcal F_j$ for $j\neq i$, then $Z_{ij} \in \mathcal F_j$ since $\mathcal F_j$ is an ideal. Therefore $Z_i \not \in \mathcal F_j$ for $j\neq i$. 

If $Z_i \oplus Z_j <_e A$, then $Z_i \oplus Z_j \in \mathcal F_k$ for some $k$. But then $i=j=k$. Therefore if $i\neq j$, then $Z_i \oplus Z_j \equiv_e A$ and $A$ is $(n-1)$-join reducible witnessed by $Z_1,\dots,Z_n$. 

By Lemma~\ref{lem: idealcover}, $A$ is $n$-join irreducible. Note that $\mathcal F_i \subseteq \{Z: Z\oplus Z_i <_e A\}$. On the other hand if $Z\oplus Z_i <_e A$, and $Z\oplus Z_i \in \mathcal F_j$, then $Z_i \in \mathcal F_j$, so $j=i$. Therefore $\mathcal F_i = \{Z: Z\oplus Z_i <_e A\}$. 

$(\romannumeral 3)\implies (\romannumeral 1)$: We saw above that every proper covering by $n$ ideals is generated by sets witnessing $(n-1)$-join reducibility. So it is sufficient to show that if we have two different sets of witnesses for $(n-1)$-join reducibility, they generate the same ideals.

 Let $Y_1,\dots,Y_n <_e A$ and $X_1,\dots,X_n <_e A$ be such that  $Y_i\oplus Y_j \equiv_e A$ and $X_i \oplus X_j \equiv_e A$ for $i\neq j$. Let $\mathcal F_i = \{Z : Z\oplus X_i <_e A\}$. For every $i$, there is exactly one $j$ such that $Y_i \in \mathcal F_j$. Without loss of generality suppose $Y_i \in \mathcal F_i$ and $Y_i \not \in \mathcal F_j$ for $j\neq i$. Then $\{Z: Z\oplus Y_i <_e A\} = \mathcal F_i$ as argued above.
\end{proof}
We can now precisely characterize the left halves of an Ahmad $n$-pair.
\begin{thm} \label{thm:npairchar}
    $A$ is $n$-join irreducible and $\lowww$ if and only if $\exists B_1,\dots,B_n $ such that $(A, B_1,\dots,B_n)$ form an Ahmad $n$-pair.
\end{thm}
\begin{proof}
\begin{case}{$(\impliedby)$}
If $A$ is the left half of an Ahmad $n$-pair then by Theorem \ref{thm:NPairSeq} it has an Ahmad sequence and by Theorem \ref{thm:main} it is $\lowww$. $\{Z: Z<_e A\}$ is covered by the ideals $\{Z: Z \leq_e A,B_i\}$ for $i=1,\dots,n$. Therefore $A$ is $n$-join irreducible by Lemma~\ref{lem: idealcover}.
\end{case}
\begin{case}{$(\implies)$}
Let $m\leq n$ be such that $\{Z:Z<_e A\} = \bigcup_1^m \mathcal F_i$ is a proper covering by $m$ ideals. Let $\mathcal F_i = \{Z: Z\oplus Z_i <_e A\}$ with the $Z_i$'s as in Theorem~\ref{thm:charirr}. Since $A$ is $\lowww$, by Theorem~\ref{thm:main} the family $\{Z:Z<_e A\}$ has a uniform enumeration while the parameter theorem gives us a uniform enumeration of the family of sets $\{\Gamma_n(A) \oplus Z_i\}_n$. Therefore by Lemma~\ref{lem:uniformEnum}, $ \mathcal G_i := \{\Gamma_n(A) \oplus Z_i \} \cap \{\Gamma_n(A) : \Gamma_n(A) <_e A\}$ has a uniform enumeration, hence so does $\mathcal F_i$ which is the left halves of joins of $\mathcal G$. 

Now let $B_i \leq_e 0'_e$ be the set with $B_i \not \geq_e  A$ and $\forall Z \in \mathcal F_i (Z \leq_e B_i)$ from Lemma \ref{lem:boundIncomp}. Then $(A,B_1,\dots,B_m)$ is an Ahmad $m$-pair. If $m<n$, we can extend this to an Ahmad $n$-pair by appending $0_e$'s. \qedhere
\end{case}
\end{proof}
In Theorem~\ref{thm:nIrrConst}, we build a set $A$ which is low, join reducible and $2$-join irreducible. Let $\mathcal F_0, \mathcal F_1$ be a proper covering of $\{Z: Z<_e A\}$. Then $\mathcal F_0$ cannot be an Ahmad sequence for any degree: Suppose $C$ was a degree with $\mathcal F_0$ as an Ahmad sequence. If $C <_e A$, then there is a $D$ such that $C <_e D <_e A$ by the density of the local structure. So $D \in \mathcal F_1$, but then $\mathcal F_0 \subseteq \mathcal F_1$, a contradiction! If on the other hand $C \not \leq_e A$ then $(C, A)$ is an Ahmad pair, but since $A$ is $\lowww$ this is not possible as we will see in Theorem~\ref{thm:righthigh2} below.
\begin{cor}
    There is a non-principal ideal of $\lowww$ degrees with a uniform enumeration, which is not the Ahmad sequence of a $\Sigma^0_2$ degree.
\end{cor}

\section{Right Halves of Ahmad Pairs}

In this section, we will show that the right half of an Ahmad pair has to be $\highh$. For this, we first argue that the skip of the left half is always below that of the right half. This will then allow us to show that $A^{\langle 3 \rangle} \equiv_1 \{e: \Gamma_e(A) <_e A\} \leq_e B^{\langle 2\rangle}$.
\begin{defn}
A set $G$ is an \emph{$A$-Guttridge set} if there is a computable function $f$ such that $f(x,.)$ is increasing, $\lim_s f(x,s)$ exists for every $x$, and $\langle x,y\rangle \in G\iff \exists s (f(x,s)>y$ or $f(x,s) = y$ and $x \in A)$.
\end{defn}
Note that if $\Theta$ is the Guttridge operator, then $\Theta(A)$ is an $A$-Gutteridge set and $\Theta(A) <_e A$ when $A$ is not c.e. 
\begin{lem}\label{lem:skipofpair}
    Let $A\subseteq \omega$ be any set. Then $\exists G<_e A$ with $G^\diamond \equiv_e A^\diamond$. Therefore, if $(A,B)$ is an Ahmad pair,  $A^\diamond \leq_e B^\diamond$.
\end{lem}
\begin{proof}
    This is implicit in the ``no symmetric Ahmad pair'' argument \cite{ahmad}.  Let $G = \Theta(K_A)$ be a $K_A$-Guttridge set, where $\Theta$ is the Guttridge operator, with $f$ being the witnessing computable function. Since $K_A \leq_e A$, we have $G <_e A$. 
    
    Note that $x \not \in K_A$ precisely when there is a $y$ such that $\langle x,y\rangle \not \in G$ and $\exists s (f(x,s) = y)$. Therefore, $\overline{K_A} \leq_e \overline G \leq_e G^\diamond$. Hence \  $A^\diamond \leq_e G^\diamond$. Now if $(A,B)$ is an Ahmad pair, then $G \leq_e B$ and so $A^\diamond \leq_e B^\diamond$.
\end{proof}
Note that in the above proof, it was important to consider $\Theta(K_A)$, and not $\Theta(A)$. It is unclear if the Guttridge operator is degree theoretic. 

To extend this to right halves of Ahmad $n$-pairs, we need the following:
\begin{lem}\label{lem:intervalcodejump}
    Let $A\subseteq \omega$ be a set with a good approximation and let $C<_e A$. There is a set $D$ such that $C\leq_e D <_e A$ and $D^\diamond \equiv_e A^\diamond$.
\end{lem}
\begin{proof}
      We construct an enumeration operator $\Theta$, such that $\Theta(A) \oplus C <_e A$ and $\Theta(A)^\diamond \geq_e A^\diamond$. Let $A_s$ be a good approximation to $A$ and $C_s$ a correct approximation to $C$, with respect to the good approximation $A_s$. The requirements are:
    \begin{align*}
        \mathcal N_e &: \Gamma_e(\Theta(A)\oplus C) \neq A, \\
        \mathcal H_e &: e \in \Gamma_e(A) \iff \Theta(A)^{[e]} = \omega.
    \end{align*}
    At stage $s=0$, let $\Theta_s = \emptyset$. At stage $s+1$ we do the following:
    \begin{enumerate}
        \item For every $e\leq s$, let $l_{e,s} = l(A_s, \Gamma_{e,s}(\Theta_s(A_s)\oplus C_s))$. Then for every $x< l_{e,s}$, if $x \in \Gamma_{e,s}((\Theta_s(A_s)^{[< e]} \cup \N^{[\geq e]})\oplus C_s)$, let $\langle x, E\oplus F\rangle \in \Gamma_{e,s}$, be the least axiom that witnesses this. Enumerate $\angles{y}{A_s}$ into $\Theta$ for every $y \in E^{[\geq e]}$.
        \item If $e \in \Gamma_e(A_s)$, then for every $y \in \omega$ enumerate the axiom $\angles{\angles{e}{y}}{A_s}$ into $\Theta$.
    \end{enumerate}
    \begin{claim}
        Every column $\Theta(A)^{[e]}$ is finite or $\omega$, and every $\mathcal N_e$ enumerates only finitely many elements into $\Theta(A)$.
    \end{claim}
    \begin{proof}
       We prove the claim by induction on $e$. Let $G_A = \{s: A_s \subseteq A\}$ be the set of  good stages.  Only actions taken at good stages affect $\Theta(A)$. Suppose the claim is true for $i< e$. $\{l_{e,s}\}_{s \in G_A}$ is unbounded if and only if $A= \Gamma_e (\Theta(A)\oplus C)$. 
       
       By construction of $\Theta$, if $\{l_{e,s}\}_{s\in G_A}$ is unbounded, we have ensured that $\Gamma_e(\Theta(A)\oplus C)  = \Gamma_e((\Theta(A)^{[< e]} \cup \mathbb N^{[\geq e]})\oplus C)$. The induction hypothesis ensures that $\Theta(A)^{[< e]}$ is a c.e.\ set. Therefore if $l_{e,s}$ is unbounded, we would have $A \leq_e C$. 

       Hence we can assume that $l_{e,s}$ is bounded, and so $\mathcal N_e$ only enumerates finitely many elements into $\Theta(A)$. $\mathcal H_e$ either enumerates no elements into $\Theta(A)$ or makes the column $\Theta(A)^{[e]}$ have all $1$'s. So $\Theta(A)^{[e]}$ is finite or $\omega$.
    \end{proof}
    Observe that $e \in A^\diamond$ iff $e \not \in \Gamma_e(A)$ iff $\exists n (\Theta(A)^{[e]}(n) = 0)$ and so $A^\diamond \leq_e \overline{\Theta(A)}$. Therefore with $D = \Theta(A) \oplus C$, we get $C\leq_e D <_e A$ and $A^\diamond \leq_e D^\diamond$ as required.
\end{proof}
\begin{defn} We say $(A,B)$ is an Ahmad pair in the cone above $D$ if $D<_e A,B$ and $A \not \leq_e B$ and $\forall Z (D\leq_e Z <_e A \implies Z \leq_e B)$.
\end{defn}
\begin{lem}\label{lem:pairCone}
If $(A,B)$ is an Ahmad pair in the cone above $D$, then $A$ is $\lowww$ and $B$ is $\highh$.
\end{lem}
\begin{proof}
By Lemma \ref{lem:intervalcodejump}, there is a $Z \in [D , A)$ such that $Z^\diamond \geq_e A^\diamond$. Therefore $B^\diamond \geq_e A^\diamond$. Let $\mathcal F = \{e: \Gamma_e(A) \leq_e B\}$. By Lemma~\ref{lem:incompbelowskip}, $A^{\langle 3 \rangle} \leq_1 \mathcal F$.

To show that $B$ is $\highh$, we will argue that $B^{\langle 2 \rangle}\geq_e  \mathcal F$. The following analysis is similar to the proof of Theorem~\ref{thm:main}: $e \in \mathcal F$ if and only if $\exists n (\Gamma_e(A) = \Gamma_n(B))$ if and only if
    \[ \exists n \neg \exists m \;( \Gamma_e(A)(m) \neq \Gamma_n(B)(m)).\]
    Note that $m \in \Gamma_e(A)$, $m \not \in \Gamma_e(A)$, $m \in \Gamma_n(B)$ and $m \not \in \Gamma_n(B)$ are all either $\leq_1 A'$ or $\leq_1 B'$. 
    
    Since $A$ and $B$ are $\Sigma^0_2$, $A' \leq_e A^\diamond$ and $B' \leq_e B^\diamond$. Therefore $B',A' \leq_e B^\diamond$. By Lemma~\ref{lem:andorproj}, the expression $\exists m \;( \Gamma_e(A)(m) \neq \Gamma_n(B)(m))$ is $\leq_e B^\diamond$ while its negation is $\leq_1 B^{\langle 2 \rangle}$. Using Lemma~\ref{lem:andorproj} again, we see that $\mathcal F \leq_e B^{\langle 2 \rangle}$.
\end{proof}
Ahmad \cite{ahmad} showed that there are no sets $A,B$ with $A \not \equiv_e B$ such that $\{Z:Z<_e A\} = \{Z: Z<_e B\}$. We can extend this below.
\begin{cor}
There are no sets $A,B,D$ with $D<_e A,B$ and $A \not \equiv_e B$ such that $\{Z: D \leq_e Z <_e A\} = \{Z: D\leq_e Z <_e B\}$.
\end{cor}
\begin{proof} Suppose $A, B, D$ were such that $D<_e A,B$ and $\{Z: D \leq_e Z <_e A\} = \{Z: D\leq_e Z <_e B\}$. Then $(A,B)$ and $(B,A)$ would be Ahmad pairs in the cone above $D$, and so by Lemma~\ref{lem:pairCone} $A$ would be $\lowww$ and $\highh$, a contradiction! 
\end{proof}
If $(A,B_0)$ is an Ahmad pair, then $(A,B_0,B_1)$ is an Ahmad $2$-pair for any $B_1 \not \geq_e A$, in particular $B_1$ may be low, or even $0_e$. However, we have the following:

\begin{thm}\label{thm:righthigh2}
    If $(A,B_0,\dots,B_{n-1})$ is an Ahmad $n$-pair and $B_i$ is essential, then $B_i$ is $\highh$. In particular if $(A,B_0,\dots,B_{n-1})$ is a proper Ahmad $n$-pair, then $B_j$ is $\highh$ for every $j<n$.
\end{thm}
\begin{proof}
   Since $B_i$ is an essential right half, there is a $D \leq_e A,B_i$ such that $D \not \leq_e B_j$ for $j\neq i$. But then $(A,B_i)$ is an Ahmad pair in the cone above $D$. Therefore by Lemma~\ref{lem:pairCone}, $B_i$ is $\highh$.
\end{proof}
 Using $\mathcal K$ pairs, Ganchev and Sorbi \cite{initSeg} construct an initial segment of high degrees below $0'_e$. Let $\mathcal I := \{Z: 0_e <_e Z <_e X\}$ be such an initial segment. No $B \in \mathcal I$ can be the right half of an Ahmad pair: For a non-c.e.\ set $A$, let $C <_e A$ be non-c.e.\ (which exists by downward density). If $C \leq_e B$ then $C$ must be high, but then so is $A$. Therefore $(A,B)$ is not an Ahmad pair.
\begin{cor}
    There is a high degree which does not bound the right half of an Ahmad pair.
\end{cor}
\begin{defn}
  For a computable function $f$, we call the family $\{\Gamma_{f(n)}(C)\}_n$ an \emph{Ahmad sequence for $A$ relative to $C$} if $\{\Gamma_{f(n)}(C)\}_n = \{Z: Z<_e A\}$.
\end{defn}

\begin{lem}\label{lem:relSeq}
    If $(A,B)$ is an Ahmad pair, then $A$ has an Ahmad sequence relative to $B$.
\end{lem}
\begin{proof}
    Consider the set $X = \{e:\exists i \forall x  ( x \in \Gamma_e(B) \iff x \in \Gamma_i(A))\}$. Observe that $\{\Gamma_e(B) : e \in X\} = \{Z : Z<_e A\}$ since $(A,B)$ is an Ahmad pair. By Lemma~\ref{lem:skipofpair}, $B^\diamond \geq_e A^\diamond$, and $B'\equiv_e B^\diamond$ since $B$ is $\Sigma^0_2$. Replicating the analysis in Theorem~\ref{thm:main}, $\exists x \; ( x \in \Gamma_e(B) \iff x \not \in \Gamma_i(A))$ is $\leq_e B^\diamond$. Therefore its negation is below $B^{\langle 2 \rangle}$, and by Lemma~\ref{lem:andorproj}, $X \leq_e B^{\langle 2 \rangle}$. But then $X \leq_1 K_{B^{\langle 2 \rangle}} = \overline{B^{\langle 3\rangle}}$.
    
    By Lemma \ref{lem:pi3complete}, for every $e \in \omega$, we can uniformly find a set $X_{e} \leq_e B$ such that $X_{e}^{[i]} = \omega$ for some $i$ (and finite otherwise) iff $e \in X$ and $X_{e}^{[i]}$ is finite for every $e\in \omega$ iff $e \not \in X$. 
   
  Let $f$ be the computable function defined as $\Gamma_{f(e,i)}(B) = \Gamma_{e}(B)\cap X_{e}^{[i]}$. Then $\{\Gamma_{f(m)}(B)\}_m$ is an Ahmad sequence for $A$ relative to $B$.
\end{proof}
The proof here does not use `c.e.\ relative to' notions like in Lemma~\ref{lem:uniformEnum} since $0'_e$ is total, while $B$ may not be. Note that if $C \geq_e B$, then by the parameter theorem, if $A$ has an Ahmad sequence relative to $B$, it also has an Ahmad sequence relative to $C$.

Similarly, we now state the more general version of Theorem~\ref{thm:main}. 
\begin{thm}\label{thm:maingeneralized}
The following are equivalent: 
\begin{enumerate}
\item  $A$ has an Ahmad sequence $\{\Gamma_{f(n)}(C)\}_n$ relative to $C$.
\item $A$ is $\lowww$, $C$ is $\highh$ and either $A\leq_e C$ or $(A,C)$ is an Ahmad pair.
\end{enumerate}
\end{thm}

\begin{proof}
\begin{case}{$(\romannumeral 1) \implies (\romannumeral 2)$}
    If $A$ has an Ahmad sequence relative to $C$, then by Lemma~\ref{lem:skipofpair}, $C^\diamond \geq_e A^\diamond$. Therefore the proof of Theorem~\ref{thm:main} shows that $X_A \leq_e C^{\langle 2 \rangle}$. Since $X_A \equiv_1 A^{\langle 3\rangle}$, we must have $A^{\langle 3 \rangle} \leq_e C^{\langle 2 \rangle}$. Since $A$ and $C$ are $\Sigma^0_2$, the jump and skip coincide, and so this can only happen if $A$ is $\lowww$ and $C$ is $\highh$. 
    
    Moreover it is clear that $C \geq_e A$ or $(A,C)$ is an Ahmad pair, as otherwise, we cannot have $C$ indices for $\{Z: Z<_e A\}$.
\end{case}
\begin{case}{$(\romannumeral 2) \implies (\romannumeral 1)$}
 The proof in Theorem~\ref{thm:main} uses the fact that $C \geq_e A$ and $X_A \equiv_1 A^{\langle 3 \rangle} \leq_e C^{\langle 2 \rangle}$. Therefore if $A$ is $\lowww$, $C$ is $\highh$ and $A\leq_e C$, there is an Ahmad sequence for $A$ relative to $C$. 

 On the other hand if $(A,C)$ is an Ahmad pair, then by Lemma~\ref{lem:relSeq} there is an Ahmad sequence for $A$ relative to $C$. \qedhere
\end{case}
\end{proof}

We can strengthen Lemma~\ref{lem:boundIncomp} to get more information about the distribution of the right halves in the local structure for a particular left half.
\begin{lem}\label{lem:boundIncompHigh}
Let $f$ be computable and $\mathcal F = \{\Gamma_{f(n)}(C)\}_n$ be an ideal with $A \not \in \mathcal F$. Then there is a set $B\leq_e C$ with $A \not \leq_e B$ and $\forall X \in \mathcal F (X \leq_e B)$. Moreover, we can ensure $B^\diamond \equiv_e C^\diamond$ and given a set $D<_e C$, we can ensure $B \not \leq_e D$. 
\end{lem}
\begin{proof}
We now have two cases, either $A \not \leq_e C$ in which case we just omit the $\mathcal N$ requirement below, or $A \leq_e C$ in which case we keep the $\mathcal N$ requirement. So suppose we are in the latter (harder) case. 

Let $\{A_s\}_s$, $\{D_s\}_s$ and $\{B_{n,s}\}_s$ be correct approximations to $A$, $D$ and $\Gamma_{f(n)}(C)$ respectively, with respect to a good approximation $\{C_s\}$ to $C$. We will build an enumeration operator $\Theta$ so that $B = \Theta(C)$ will meet the requirements: 
\begin{align*}
  \mathcal{N}_e&: A \neq \Gamma_e(B), \\
  \mathcal{P}_e&: \Gamma_{f(e)}(C) =^* B^{[3e]},\\
  \mathcal{R}_e&: \Gamma_e(D) \neq B,\\
  \mathcal{H}_e&: e\in  \Gamma_e(C)\iff \forall s ( B^{[3e+2]}(s) = 1 ).
\end{align*}
At stage $s=0$, let $\Theta = \emptyset$.\\
At stage $s+1$:
\begin{enumerate}
\item 
For $e\leq s$ let  $l_{e,s} = l(A_s, \Gamma_{e,s}(B_s))$. Then for every $x < l_{e,s}$, if $x \in \Gamma_{e,s}(B_s^{[\leq e]} \cup \N^{[>e]})$, pick the least axiom $\langle x, D \rangle \in \Gamma_e$ which witnesses this. Now for every $y \in D^{[> e]}$ enumerate the axiom $\langle y, C_s \rangle$ into $\Theta$. 
\item
For $e\leq s$ and for every axiom $\langle x, D \rangle \in \Gamma_{f(e),s+1}$, enumerate the axiom $\langle \langle 3e,x\rangle, D\rangle$ into $\Theta$.
\item For $e \leq s$ let $\hat l_{e,s} = l(\Gamma_{e,s}(D_s), B_s)$. Then $\forall x<\hat l_{e,s}$ enumerate the axiom $\langle \langle 3e+1,x\rangle, C_s\cup \{x\} \rangle$ into $\Theta$.
\item For $e \leq s$, if $e \in \Gamma_{e,s}(C_s)$, enumerate the axioms $\langle \langle 3e+2,t\rangle , C_s \rangle$ for every $t\in \omega$. 
\end{enumerate}
\textbf{Verification}: Let $G= \{s: C_s \subseteq 0'_e\}$ be the set of good stages. Recall that $A = \Gamma_e(B) \iff \lim_{s\in G} l_{e,s}= \infty$. Observe that by step $(\romannumeral 2)$, we have ensured that $\Gamma_{f(e)}(C) \subseteq B^{[3e]}$.
\begin{lem} Every $\mathcal N_e$ and $\mathcal R_e$ strategy satisfies its requirement and only enumerates finitely many elements into $B$.
\end{lem}
\begin{proof}
    We prove this by induction on $e$. Suppose $\mathcal N_i$ and $\mathcal R_i$ enumerates finitely many elements into $B$ for every $i<e$. Observe that $\Gamma_e(B)\subseteq \Gamma_e(B^{[\leq e]} \cup \mathbb N^{[>e]})$. Suppose $\Gamma_e(B) = A$ and so $\{l_{e,s}\}_{s\in G}$ is unbounded. Now if $x \in \Gamma_e(B^{[\leq e]} \cup \mathbb N^{[>e]})$ witnessed by $\langle x, D \rangle \in \Gamma_e$ then $\mathcal N_e$ will dump $D^{[>e]}$ into $B$ and so $x \in \Gamma_e(B)$. Therefore when $\Gamma_e(B) = A$, we have $\Gamma_e(B) = \Gamma_e(B^{[\leq e]} \cup \N^{[>e]})$.

    Let $3m+i\leq e$ for $i<3$. Then we have the following cases:
    \begin{enumerate}
        \item If $i=0$, then $\mathcal P_m$ and $\mathcal N_{m'}$ for $m'<3m\leq e$ are the only strategies that enumerates elements into $B^{[3m]}$ and $\mathcal P_m$ is successful in coping in $\Gamma_{f(e)}(C)$ into this column up to finite difference by the induction hypothesis.
        \item If $i=1$, then $\mathcal R_m$ and $\mathcal N_{m'}$ for $m' < 3m+1\leq e$ are the only strategies that enumerate elements into $B^{[3m+1]}$. By the induction hypothesis, this column is then finite.
        \item If $i=2$, then again the $\mathcal N_{m'}$ for $m'<3m+2 \leq e$ contribute only finitely many elements. The only other requirement which can act is $\mathcal H_m$, and this either makes the entire column $B^{[3m+2]} = \omega$ or does not enumerate any elements at all. So this column is either finite or equals $\omega$.
    \end{enumerate}  
    Now $\{l_{e,s}\}_{s\in G}$ is unbounded $\iff$ $\Gamma_e(B^{[\leq e]} \cup \mathbb N^{[>e]}) = \Gamma_e(B) = A \iff$  $A \leq_e \bigoplus_{i\leq e} B^{[i]} \equiv_e \bigoplus_{i \leq e/3} \Gamma_{f(i)}(C)$. But since $A \not \in \mathcal F$, this is not possible and so $\{l_{e,s}\}_{s\in G}$ must be bounded. Therefore $\mathcal N_e$ is met and it only enumerates finitely many elements into $B$.

    If $\Gamma_e(D) = B$, then $\hat l_{e,s}$ would be unbounded, and $\Theta(C)^{[3e+1]} =^* C$ by the analysis above. But then $D \geq_e C$, a contradiction. Therefore $\Gamma_e(D) \neq B$, and since $\hat l$ is bounded at good stages, $R_e$ only enumerates finitely many elements into $B$
\end{proof}
Therefore since all the $\mathcal N$ requirements enumerate finitely many elements into any column, the $\mathcal P_e$ requirements are successful in ensuring that $B^{[3e]} =^* \Gamma_{f(e)}(C)$.

Lastly to see that the $\mathcal H_e$ requirements are met, observe that $e \in C^\diamond$ precisely when there is a $0$ in the column $B^{[3e+2]}$. But this question is enumeration below the complement of $B$ and so $C^\diamond \leq_e B^\diamond$.
\end{proof}
We can now get a right half which is not high.
\begin{cor}
There are sets $A,B$ with $B' <_e 0''_e$ such that $(A,B)$ is an Ahmad pair.
\end{cor}
\begin{proof}
There is a $\highh$ but non-high c.e.\ Turing degree $C$ which bounds a low c.e.\ Turing degree $D$. Then $\iota(C)$ is $\Sigma^0_2$, $\highh$ and not high while $\iota(D)$ is $\Sigma^0_2$ and low since the map $\iota$ respects the jump. Join irreducibility being downwards dense \cite{kentBounding}, there is a non-c.e.\ $A \leq_e \iota(D)$ which is join irreducible. Since $\iota(C) \geq_e A$ and is $\highh$, by Theorem~\ref{thm:maingeneralized} there is an Ahmad sequence for $A$ relative to $\iota(C)$. Then by Lemma~\ref{lem:boundIncompHigh}, there is a $B \leq_e \iota(C)$ such that $(A,B)$ is an Ahmad pair. This $B$ is $\highh$ and not high.
\end{proof}
Let $A$ be $\lowww$ and join irreducible. Consider the family of $\Sigma^0_2$ sets $\mathcal F_A = \{Y : \exists Z <_e A (Y \oplus Z \geq_e A)\}$. $\mathcal F_A$ is closed upwards, and if $C \in \mathcal{F}_A$, $(A,C)$ is not an Ahmad pair. 

If $C \not \in \mathcal F_A$, consider the ideal generated by $\{Z: Z<_e A\} \cup \{C\}$. This ideal does not contain $A$ and so by Lemma~\ref{lem:boundIncompHigh} there is a $B \geq_e C$ such that $(A,B)$ is an Ahmad pair. Therefore $\{ B : (A,B) \;\text{is an Ahmad pair}\}$ is cofinal and closed upwards in $\overline {\mathcal{F}_A}$.

If $(A,B)$ is an Ahmad pair then $A \oplus B >_e B$ and for every $C$ in the interval $[B,A\oplus B) = \{Z : B \leq_e Z <_e A\oplus B\}$, $(A,C)$ is an Ahmad pair. Therefore there is no maximal right half. Moreover every interval has a join reducible element, and so there is a join reducible right half.

By Lemma~\ref{lem:relSeq}, if $(A,B)$ is an Ahmad pair, then $A$ has an Ahmad sequence relative to $B$. Then by Lemma~\ref{lem:boundIncompHigh}, there is a $C \leq_e B$ such that $(A,C)$ is an Ahmad pair and $B \not \leq_e C$. Therefore there is no minimal right half either. 
\begin{cor}
For every Ahmad pair $(A,B)$, there is a $C <_e B$ such that $(A,C)$ is an Ahmad pairs.
\end{cor}
Using the ideas of this section, we get a $\Pi_3$ definition of $\highh$ in the local structure in the language $\{\leq_e\}$. The defining formula below says that a degree is $\highh$ if everything above it which bounds a left half also bounds a right half.
\begin{thm}\label{thm:defhigh2}
$X$ is $\highh$ if and only if $ \forall Y \geq_e X \forall A \leq_e Y(\exists B \not \geq_e A \forall Z <_e A( Z\leq_e B)) \implies (\exists C \leq_e Y, C \not \geq_e A \forall Z <_e A (Z\leq_e C))$.
\end{thm}
\begin{proof}
If $X$ is $\highh$ and $ Y \geq_e X$, then $Y$ is $\highh$. Therefore if $Y$ bounds a left half, by Theorem~\ref{thm:maingeneralized}, it also bounds a right half corresponding to that left half. 

In his seminal work, Kalimullin \cite{kPairs} introduced the notion of $\mathcal K$-pairs to define the jump operator. Fix a non-c.e.\ set $X$, and let $(A,B)$ be a $\mathcal K$-pair below $0'_e$ with $A, B \not \leq_e X$. Then $A$ is low and $(A\oplus X)' \leq_e A\oplus B \oplus X'$ by the properties of $\mathcal K$ pairs in \cite{kPairs}. But  $A\oplus B \oplus X'\leq_e X'$ since $A$ and $B$ are $\Sigma^0_2$, and so $(A\oplus X)' \leq_e X'$. 

Therefore if $X$ is not $\highh$, there is a set $A$ which is low such that $A\oplus X$ is not $\highh$. Taking $Y = X\oplus A$, we know that $Y$ bounds a left half as it bounds a low join irreducible degree \cite{kentBounding}, and such degrees are left halves by Corollary~\ref{cor:pairiff}. However $Y$ cannot bound a right half by Theorem~\ref{thm:maingeneralized}.
\end{proof}

In this paper, we have separated the right and left half in terms of jump classes. It would also be interesting to know if there is a separation based on join irreducibility. The authors in \cite{kentBounding} show that join irreducibility is downwards dense and ask whether every interval contains a join irreducible degree. A positive answer would also give a positive answer to the question below.

\begin{query} Is there an Ahmad pair $(A,B)$ such that $B$ is join irreducible?
\end{query}

Beyond downwards density, the distribution of join irreducible degrees in the local structure is unknown. Towards a negative answer to the question in \cite{kentBounding} we ask:
\begin{query} Is every join irreducible degree quasiminimal?
\end{query}
The constructions of a right half using an Ahmad sequence in this paper does not seem compatible with making the right half total.
\begin{query} Is there an Ahmad pair $(A,B)$ such that $B$ is total?
\end{query}
While we obtain a characterization of the left halves, a similar characterization of the right halves remains elusive.
\begin{query}
For every $\highh$ $B$ which bounds a $\lowww$ degree, is there an $A$ such that $(A,B)$ is an Ahmad pair?
\end{query}

\section{\texorpdfstring{A Proper Hierarchy of Ahmad $n$-Pairs}{A Proper Hierarchy of Ahmad Pairs}}
 Ahmad and Lachlan \cite{ahmad} build a low join irreducible degree. The authors in \cite{extensions} extend this to build a join reducible degree which is the left half of an Ahmad $2$-pair. In this section we extend these ideas to show that for every $n\geq 1$ there is a low $n$-join irreducible set which is $(n-1)$-join reducible (note that vacuously every set is $0$-join reducible). As a corollary, this implies that there are Ahmad $n$-pairs whose left half cannot be part of an Ahmad $(n-1)$-pair for $n>1$.
    
\begin{thm}\label{thm:nIrrConst}
   For every $n\geq1$ there is a set $A$ which is low, $n$-join irreducible and $(n-1)$-join reducible.
\end{thm}
\begin{proof}
    We build independent sets $X_0,X_1,\dots,X_{n-1}$ and let $A = \bigoplus_{i< n} X_i$. Let $Y_i \subseteq A$ be defined by $Y_i^{[j]} = X_j$ for $i \neq j< n$ and $Y_i^{[i]} = \emptyset$ where slightly abusing notation, $Y_i^{[j]}$ refers to the $j^{th}$ of $n$ columns. To make the $X_i$'s independent, we will ensure that $X_i \not \leq_e Y_i$ for every $i<n$. Then $A$ is $(n-1)$-join reducible since  $Y_i <_e A$ for every $i<n$ but $Y_i \oplus Y_j \geq_e A$ for any $i\neq j$. We construct $A$ to meet the requirements:
    \begin{align*}
        \mathcal N_k^i &: X_i\neq \Gamma_k(Y_i),\\
        \mathcal L_k^x &: \exists^\infty s\; x \in \Gamma_{k,s}(A) \implies x \in \Gamma_k(A).
    \end{align*}
     For the requirement below let $A_n$ denote $\Gamma_n(A)$. Then, for every $f = \langle f_0,\dots,f_n \rangle \in \omega^{n+1}$, $e = \langle e_{ij} \rangle_{i<j\leq n} \in \omega^{\binom{n+1}{2}}$, we require $A$ to satisfy:
    \begin{align*}    
        \mathcal P_f^e &: \forall i <j \leq n  ( A = \Gamma_{e_{ij}}(A_{f_i} \oplus A_{f_j})) \implies \exists k \leq n  ( A \leq_e A_{f_k}).
    \end{align*}
    This directly ensures that $A$ is $n$-join irreducible. To simplify the construction, we modify $\mathcal P^e_f$ to prove something slightly stronger: If each pair of the $A_{f_i}$'s join up to $A$, then we will ensure that every $X_i$ is below at least $n$ of the $A_{f_i}$'s.
    
    In other words for every $i$, there is at most one $j$ such that $X_i \not \leq_e A_{f_j}$. Then $A = \bigoplus_{i<n} X_i$ must, by the pigeon hole principle, be below at least one of the $n+1$ $A_{f_i}$'s. The precise requirement is given below:
    \begin{align*}
        \mathcal P_f^e &: \forall i<j \leq n ( A = \Gamma_{e_{ij}}(A_{f_i} \oplus A_{f_j})) \implies \\
        &\forall k<n \exists i_k \leq n\;  (X_k \leq_e A_{f_j} \; \text{for every } j \neq i_k).
    \end{align*}
    
    We build $A$ as a $\Delta^0_2$ set with $A_s$, our approximation to $A$ at stage $s$. We permanently dump certain elements into $A$ during the construction; let $\hat A$ be this c.e.\ set of dumped elements. Moreover let $\hat X_i$ for $i<n$ be defined by $\hat A = \bigoplus_{i<n} \hat X_i$. Let $\{x_{k,s}^i\}_k$ enumerate $\omega-\hat X_{i,s}$ in increasing order. For a fixed $i$, we think of the $x_{k,s}^i$ as movable markers covering $\hat X_{i,s}$'s complement. During the construction, we sometimes focus on the marker while at other times, the number the marker is currently at.
    
    Let $x^i_k = \lim_s x^i_{k,s}$ and $B_i = \{x_k^i : k \in \omega\}$ be the set of `true witnesses' of the $\mathcal N^i_k$ requirements, so $B_i = \omega - \hat X_i$. Let $B=\bigoplus_{i<n} B_i$. Note that $A \leq_e B\cap A$ since $\hat A$ is c.e.\ and for every $i$, $X_i \leq_e B_i \cap X_i$ since $\hat X_i$ is c.e.
    
    Slightly abusing notation, we use $\angles{i}{m}$ to denote position $m$ in the $i$th member of a join of $n$ sets. So, $x^i_k \in X_i$ is equivalent to $\angles{i}{x^i_k} \in A$. We call $\angles{i}{x^i_k}$ the witness of $\mathcal N^i_k$. Note that during the construction, changes made to $A$ implicitly change the $X_i$'s and vice versa.

    We now describe the construction of $A$. Order the requirements in order type $\omega$ as $S_0, S_1, S_2,\dots$ such that, for every $i <n$ and $k \in \omega$, if $S_\alpha = \mathcal N^i_k$ and $S_\beta = \mathcal N^i_{k+1}$, then $\alpha < \beta$. We let $S_\alpha < S_\beta$ mean $S_\alpha$ has higher priority than $S_\beta$.
    
    At $s = 0$ let $A_s = \emptyset$ and $\hat A_s = \emptyset$. Moreover all the markers $x^i_k$ are inactive for every $i<n$ and $k\in \omega$.
    
    At stage $s+1$, for every $i<n$, activate the marker $x^i_{s}$. Then for each $t\leq s$, we have a substage where $S_t$ gets to act. Our approximations are dynamically updated at each substage of stage $s+1$. $S_t$ acts as follows depending on its type:
    \begin{enumerate}
    \item $S_t =\mathcal N_k^i$ : If $x_{k,s+1}^{i} \not \in \Gamma_{k,s}(Y_{i,s})$, put $x_{k,s+1}^i$ into $X_{i,s+1}$. 
    If $x_{k,s+1}^{i} \in \Gamma_{k,s}(Y_{i,s})$, extract $x^{i}_{k,s+1}$ out of $X_{i,s}$ and dump all currently active witnesses of $\mathcal N$ requirements of lower priority than $\mathcal N^i_k$ into $A$. Observe that pulling $x^i_{k,s+1}$ out of $X_{i,s}$ will not affect $\Gamma_k(Y_{i,s+1})$.
    
    \item $S_t=\mathcal L_k^x$ : If $x \in \Gamma_{k,s+1}(A_{s+1})$ with $\langle x, D \rangle$ the oldest axiom witnessing this i.e.\ the one that has been valid the longest, then dump the elements of $D$ into $\hat A_{s+1}$ except for witnesses of $\mathcal N < \mathcal L^x_k$ requirements. Do this at most once per axiom.
    
    \item $S_t=\mathcal P^e_f$ : It builds a local copy of operators $\Delta^k_i$ for $k <n$ and $i<n+1$, and initializes them all to be $\emptyset$. The goal is to have $X_k =^* \Delta_i^k(A_{f_i})$.

    For every $k$ and for every $x \in \hat X_{k,s}$, enumerate the axiom $\angles{x}{\emptyset}$ into $\Delta^k_i$ for $i\leq n$. Note that such $x$'s will never leave $X_k$.
    
    Let the outcome of marker $x^k_p$ at stage $s+1$ for $\mathcal P^e_f$ be denoted by $O^e_{f,s+1}(x^k_{p})$. The outcomes of all active markers relative to $\mathcal P^e_f$ are initially set to $\infty$ when first activated. We will see that a number will only change its outcome once. If a marker inherits the number of another marker, it also inherits the corresponding outcome.
    \begin{enumerate}
    \item If $\mathcal N^k_p > \mathcal P^e_f$ and some active witness $\langle k, x^k_{p,s+1}\rangle \in A_{s+1}$ as well as $\langle k, x^k_{p,s+1}\rangle \in \Gamma_{e_{ij},s+1}(A_{f_i,s+1} \oplus A_{f_j,s+1})$ for $i<j \leq n$, then enumerate axioms $\angles{x^k_{p,s+1}}{A_{f_i,s+1}}$ into $\Delta^k_i$ for $i \leq n$. Then dump all currently active witnesses for $\mathcal N > \mathcal N^k_p$ requirements into $\hat A_{s+1}$. 
    
    Therefore if $\langle k, x^k_{p,s+1} \rangle \not \in \hat A$, then $\langle k, x^k_{p,s+1}\rangle$ can leave $\Gamma_{e_{ij}}(A_{f_i} \oplus A_{f_j})$ only due to $\langle k, x^k_{p,s+1}\rangle$ leaving $A$ (a witness of higher priority leaving $A$ would cause $x^k_{p,s+1}$ to get dumped, while lower priority witnesses have already been dumped, and so cannot leave). Never enumerate an axiom in part (a) into a (particular local) $\Delta^k_i$ for the number $x^k_{p,s+1}$ ever again. 
    
    \item Suppose $\mathcal N^k_p > \mathcal P^e_f$ and some active witness $x^k_{p,s+1}$ had an axiom enumerated for it in part (a) at some stage $u<s$, but $\langle k, x^k_{p,s+1}\rangle\not \in A_{s+1}$ and $\langle k, x^k_{p,s+1}\rangle \not \in \Gamma_{e_{ij},s+1}(A_{f_i,s+1} \oplus A_{f_j,s+1})$ for $i<j\leq n$. Since the computations of the $\Gamma_{e_{ij}}$'s have changed, for every $i, j$ with $i\neq j$ it cannot be the case that $A_{f_i,u}\subseteq A_{f_i,s+1}$ and $A_{f_j,u}\subseteq A_{f_j,s+1}$. Therefore there is at most one $i\leq n$ such that $x^k_{p,s+1} \in \Delta^k_i(A_{f_i,s+1})$.
    
    Let $O^e_{f,s+1}(x^k_p) =i$, if $x^k_{p,s+1}\in \Delta_i^k(A_{f_i,s+1})$. If there is no such $i$, let $O^e_{f,s+1}(x^k_p) =-1$. Note that if $O^e_{f,s+1}(x^k_p) =i$, then $\Delta^k_j(A_{f_{j,s+1}})(x^k_{p,s+1}) = X_k(x^k_{p,s+1})$ for $0\leq j\leq n$, $j\neq i$. 
    
    Therefore to meet this requirement, we need to ensure two outcomes $i\neq j$ cannot occur infinitely often. To accomplish this we `backward dump' as follows: Dump into $\hat{X}_k$ every $x^k_{q,s+1}$ for $q<p$ with $\mathcal N^k_q> \mathcal P^e_f$ and such that: 
    \begin{enumerate}
        \item $x^k_q$ has higher outcome - $O^e_{f,s+1}(x^k_q) >i$.
        \item We believe that we are in the cofinite tail for higher priority $\mathcal P$ requirements - For every $\mathcal P^{e'}_{f'} < \mathcal P^e_f$, we have $O^{e'}_{f',s+1}(x^k_p) = O^{e'}_{f',s+1}(x^k_q)$. We will see during the verification that this will be sufficient as we will ensure all $\mathcal P$ requirements will have cofinitely many witnesses with the same outcome.
    \end{enumerate}
    This ends the construction.
    \end{enumerate}
    \end{enumerate}
    \begin{lem}
        For every $i<n$ and $k\in \omega$, the marker $x^i_k$ settles down, i.e.\ $x^i_k = \lim_s x^i_{k,s}$ exists, and the limit $\angles{i}{x^i_k}$ moves in/out of $A$ only finitely many times. 
    \end{lem}
    \begin{proof}
    We prove the claim by induction. Suppose every marker $x$ of $\mathcal N < \mathcal N^i_k$ settles down and enters or exits $A$ finitely often. Let $s$ be a large enough stage such that all such markers have settled down and will never enter or exit $A$ again. The marker $x^i_k$ can only be moved by requirements of higher priority than $\mathcal N_k^i$ when they dump elements into $\hat A$. Let $S_t$ be such a higher priority requirement:
    \begin{enumerate}
        \item $S_t = \mathcal N^j_e$: After stage $s$, $S_t$ cannot move $x^i_k$ since its own marker has settled and no longer enters/exits $A$.
        \item $S_t = \mathcal L_e^x$: Suppose after stage $s$, $S_t$ moves the marker $x^i_k$, by dumping elements into $\hat A$ to keep $x \in \Gamma_e(A)$. This requirement is then permanently satisfied and will never act again: No witnesses of $\mathcal N <S_t$ requirements will move their witness out of $A$ after stage $s$, while $S_t$ can dump in all the other witnesses that it needs to keep $x$ in $\Gamma_e(A)$. Therefore after stage $s$, $\mathcal L_e^x$ moves marker $x^i_k$ at most once.
        \item $S_t = \mathcal P^e_f:$ By hypothesis, after stage $s$ we have the following cases.
        \begin{enumerate}
            \item Let $y$ be a witness of $\mathcal N < \mathcal N^i_k$. Suppose there is a $u>s$ such that, $y$ is in $A_u$ and each of $\Gamma_{e_{ij},u}(A_{f_i,u} \oplus A_{f_j,u})$ for $i<j \leq n$, and this is the first time $y$ is in all of these sets.
            
            Then $\mathcal P^e_f$ will dump the currently active witnesses of $\mathcal N' > \mathcal N$ requirements. Therefore, by induction hypothesis, there is no $\mathcal N'$ requirement such that $\mathcal N < \mathcal N' < \mathcal N^i_k$. Hence $\mathcal P^e_f$ will only move $x^i_k$ at most once after stage $s$, for the sake of witnesses of $\mathcal N<\mathcal N^i_k$ requirements.
            \item Suppose there is a stage $u>s$ where for some $l > 0$, $O^e_{f,u}(x^i_{k}) >O^e_{f,u}(x^i_{k+l}) $. If $\mathcal P^e_f$ `backward dumps' $x^i_{k,u}$, then we argue below that the outcome of marker $x^i_k$ can only decrease, i.e.\ $O^e_{f,u+1}(x^i_{k})\leq O^e_{f,u}(x^i_{k+l})$.

            Let $\mathcal P_0,\dots \mathcal P_\alpha$ enumerate all the $\mathcal P$ requirements of higher priority than $\mathcal P^e_f$ in decreasing order of priority, so $\mathcal P_0 < \dots < \mathcal P_\alpha$. Then the outcomes of the markers $\{x^i_n\}_{n <u}$ with respect to $\mathcal P_0$ at stage $u$ must be weakly increasing (as a consequence of backward dumping). Moreover markers $x^i_k$ and $x^i_{k+l}$ must have the same outcome for $\mathcal P_0$ at stage $u$, let $n_0, m_0$ be the interval of markers $\{x^i_n\}_{n_0 \leq n \leq m_0}$ which have this outcome for $\mathcal P_0$ at stage $u$.

            Now the outcomes of $\{x^i_n\}_{n_0 \leq n \leq m_0}$ for $\mathcal P_1$ must be weakly increasing as $\mathcal P_1$ is free to backward dump in this interval without being blocked by $\mathcal P_0$. So we shrink our interval to $n_1 \geq n_0$ and $m_1 \leq m_0$ such that the outcomes of $\{x^i_n\}_{n_1 \leq n \leq m_1}$ are all the same for $\mathcal P_1$ at stage $u$, and this outcome is equal to the outcomes of $x^i_k$ and $x^i_{k+l}$ for $\mathcal P_1$ at stage $u$.

            Proceeding this way, we get a non-decreasing sequence $n_0\leq n_1\leq \dots \leq n_\alpha$ and a non-increasing sequence $m_0 \geq m_1 \dots \geq m_\alpha$, such that $n_\alpha \leq k < k+l \leq m_\alpha$. So finally we have $O^e_{f,u+1}(x^i_{k})\leq O^e_{f,u}(x^i_{k+l})$ since the marker $x^i_k$ inherits the outcome of a larger marker, and this outcome is at most $O^e_{f,u}(x^i_{k+l})$ as $\mathcal P^e_f$ is free to `backward dump' markers in $\{x^i_n\}_{n_\alpha \leq n \leq m_\alpha}$ with larger outcomes relative to $\mathcal P^e_f$.

            After finitely much dumping, $\mathcal P^e_f$ in isolation will no longer have cause to move marker $x^i_k$ via `backward dumping' since $O^e_{f}(x^i_{k})$ can only decrease to $-1$ after which it is not longer `backward dumped'. Moreover if a marker has outcome $\infty$, and then moves to step $(3 \romannumeral 2)$, then its outcome is only lowered.
            
            The only trouble occurs if marker $x^i_k$ is moved by another requirement, which then causes its outcome for $\mathcal P^e_f$ to increase. As argued above, after stage $s$, this can happen at most once for each higher priority $\mathcal N$ and $\mathcal L$ requirements, as well as $\mathcal P$ requirements via `forward dumping'. 
            
            On the other hand, we explicitly prohibit backward dumping by $\mathcal P' > \mathcal P^e_f$ requirements if they raise outcomes for $\mathcal P^e_f$ of witnesses. There are only finitely many $\mathcal P' < \mathcal P^e_f$ requirements. We can argue by induction that each of them `backward dumps' $x^i_k$ finitely often and does not increase the outcome of $x^i_k$ for $\mathcal P'' < \mathcal P'$ requirements. So every such $\mathcal P' <\mathcal P^e_f$ will eventually stop moving $x^i_k$ after some finite stage.
        \end{enumerate}  
    \end{enumerate}
    Let $t>s$ be a stage such that the $x^i_k$ marker has settled down. After stage $t$, $\mathcal N^i_k$ can successfully diagonalize: it first puts $\angles{i}{x^i_{k,t}}$ into $A_t$ and if it ever sees $\angles{i}{x^i_{k,t}} \in \Gamma_{k,u}(Y_{i,u})$ for some $u>t$, it can permanently keep it in by dumping lower priority witnesses, while by hypothesis higher priority witnesses will no longer exit $A$. In this case $x^i_{k,u}$ is extracted and never enters $A$ again. 
    \end{proof}
   \begin{lem}
       Every requirement is met.
   \end{lem} 
   \begin{proof}
    Let $x^i_k = \lim_s x^i_{k,s}$ be the `true witness' for $\mathcal N^i_k$, and let $\mathcal O^e_f(x^i_k) = \lim_s O^e_{f,s}(x^i_k)$ be its true outcome. Fix a requirement $S_t$ and let $s$ be a large enough stage that all markers of $\mathcal N< S_t$ requirements have settled down (will no longer be dumped) and the corresponding numbers are in or out of $A$ for the last time. We consider the following cases.
    \begin{enumerate}
    \item $S_t=  \mathcal N_k^i$: Let $u>s$ be large enough such that marker $x^i_{k,u} = \lim_{v} x^i_{k,v}$. Using $x^i_{k,u}$, the diagonalization of $\mathcal N_k^i$ will be successful.
    
    \item $S_t = \mathcal L^x_k$: Suppose $\exists^\infty u (x\in \Gamma_{k,u}(A_u))$. Then for $u>s$ if $\mathcal L_k^x$ dumps elements into $A$ at stage $u$,  then $x \in \Gamma_k(A)$ since none of the non-dumped witnesses will leave $A$ after stage $u$.
    
    \item $S_t = \mathcal P_f^e$: Suppose $\forall i<j \leq n ( A = \Gamma_{e_{ij}}(A_{f_i} \oplus A_{f_j}))$. Fix a $k<n$. We will show that there is an $i \in \{-1,0,\dots,n,\infty\}$ such that $\forall^\infty n \;O^e_f(x^k_n) = i$. To see this, suppose $\mathcal P_f^e$ was the highest priority $\mathcal P$ requirement for which this were false, and let $i$ be the least such that $\exists^\infty n \; O^e_f(x^k_n) = i$. Moreover let $N$ be large enough such that $\{O^{e'}_{f'}(x^k_n)\}_{n>N}$ is constant for every $\mathcal P^{e'}_{f'} < \mathcal P^e_f$. Suppose there were numbers $n_0$, $n_1$ and $n_2$ such that $N < n_0 < n_1 < n_2$ for which $ O^e_{f}(x^k_{n_0}), O^e_{f}(x^k_{n_2}) = i$ while $ O^e_{f}(x^k_{n_1}) = j >i$.  
    
    Let $u>s$ be a large stage such that markers $x^k_{n_0}$, $x^k_{n_1}$ and $x^k_{n_2}$ have settled down and achieved their true outcome by stage $u$ relative to every $\mathcal P \leq \mathcal P^e_f$ requirement (once a marker has settled down, its outcome can change at most once). Then at stage $u$, $\mathcal P^e_f$ would dump $x^k_{n_1,u}$ on behalf of $x^k_{n_2,u}$ as it is not blocked by any $\mathcal P< \mathcal P^e_f$ requirement. But this contradicts the assumption that $x^k_{n_2}$ has settled down by stage $u$. 

    Therefore $\forall^\infty n \;O^e_f(x^k_n) = i$. We can now argue that for $j\neq i, X_k =^* \Delta^k_j(A_{f_j})$. Note that $\hat{X_k} \subseteq X_k, \Delta^k_j(A_{f_j})$ for all $j\leq n$. On the other hand if a `true witness' $\angles{k}{x_e^k} \in A$ and $\angles{k}{x_e^k} \in \Gamma_{e_{ij}}(A_{f_i} \oplus A_{f_j})$ for every $i<j \leq n$, then $x_e^k \in \Delta^k_j(A_{f_j})$ for every $j\leq n$. If instead $x_e^k \not \in X_k$, we have the following two cases:
    \begin{enumerate}
    \item $x^k_e$ never had an axiom enumerated for it in (3a). Then $x^k_e \not \in \Delta^k_j (A_{f_j})$ for all $j\leq n$. 
    \item If $O^e_f(x^k_e) = i < \infty$, then $x_e^k \not \in \Delta^k_j(A_{f_j})$ for $j\leq n$ and $j\neq i$ as argued in step $(3\romannumeral 2)$. 
    \end{enumerate}
    Therefore we have $X_k \leq_e A_{f_j}$ for $j\neq i$ since cofinitely many of the $\{x^k_l\}_l$ have outcome $i$ and for all $j\leq n$ with $j\neq i$, if $O^e_f(x^k_l)= i$, then $\Delta^k_j(A_{f_j})(x^k_l) = X_k(x^k_l)$. \qedhere
    \end{enumerate}
    \end{proof}
    So $A$ is low, $(n-1)$-join reducible and $n$-join irreducible as required. \qedhere
\end{proof}
\begin{cor}\label{cor:npairs} For every $n>1$ there is an Ahmad $n$-pair $(A,B_1,\dots,B_n)$ such that $A$ is not the left half of an Ahmad $(n-1)$-pair.
\end{cor}
\begin{proof}
    Let $A$ be as in \ref{thm:nIrrConst}. Since $A$ is not $(n-1)$-join irreducible, it cannot be the left half of an Ahmad $(n-1)$-pair. However since $A$ is $n$-join irreducible and low, by Theorem \ref{thm:npairchar}, it is the left half of an Ahmad $n$-pair.
\end{proof}

\section{\texorpdfstring{Consequences for the $\forall \exists$-theory}{Consequences for the AE-theory}}

Using the results in this paper, we can now decide the truth or falsity of some new $\Pi_2$ sentences in the language of $\{\leq_e\}$ in the local structure. 

Deciding $\Pi_2$ sentences in the language $\{\leq_e\}$ is equivalent to deciding the following algebraic problem: Let $\mathcal P$ be a finite partial order and let $\mathcal Q_0, \dots, \mathcal Q_n$ be finite partial orders extending $\mathcal P$. Can any embedding of $\mathcal P$ into the  $\Sigma^0_2$ enumeration degrees be extended to an embedding of at least one of the $\mathcal Q_i$'s?

As mentioned in the introduction, the extension of embeddings problem is a subproblem of the $\forall \exists$-theory and corresponds to $n=0$ in the algebraic problem above. This was fully solved for the $\Sigma^0_2$ e-degrees by Lempp, Slaman and Sorbi \cite{extemb}.

An upcoming paper \cite{onePtExt} solves a different subproblem of the $\forall \exists$-theory which corresponds to limiting $\mathcal P$ to be an antichain, and each of the $\mathcal Q_i$'s being one point extensions below some elements of $\mathcal P$. 

Using the stronger results in this paper, we can now decide a larger class of sentences about the $\forall \exists$-theory. We state these new consequences below, limiting ourselves to the qualitatively new facts (presented in their simplest form, and not in full generality).
\begin{thm}\label{thm:AE}
The following $\Pi_2$ statements in the language of $\{\leq_e\}$ are true in the $\Sigma^0_2$ enumeration degrees.
\begin{enumerate}
\item $\forall A_0,B_0,A_1,B_1 \exists Z$ such that if ($B_0 \leq_e A_1$ and $A_0  \not \leq_e B_0$, $A_1 \not \leq_e B_1$), then ($Z<_e A_0$ and $Z \not \leq_e B_0$) or ($Z <_e A_1$ and $Z \not \leq_e B_1$).
\item $\forall \hat A, A, B_0,\dots, B_n, C_0,\dots,C_m \exists Z$ such that if ($\hat A \leq_e A,B_0$ and $\hat A \not \leq_e B_i$ for $i\neq 0$ and $A$, $B_0,\dots, B_n$, $C_0,\dots,C_m$ are all incomparable), then ($Z <_e A$, $Z \not \leq_e B_i$ for every $i$) or ($Z <_e B_0$, $Z \not \leq_e C_i$ for every $i$).
\item $\forall A,B,C,D_0,D_1 \exists Z$ such that if ($A$, $B$, and $C$ are incomparable, $D_0 <_e A,B$ and $D_1 <_e B,C$), then ($D_0 <_e Z <_e A$ and $Z \not \leq_e B$) or ($D_1 <_e Z <_e B$ and $Z \not \leq_e C$). 
\end{enumerate}
\end{thm}
\begin{proof} We consider each sentence in turn.
\begin{enumerate}
\item If $A_0 >_e B_0$ or $A_1 >_e B_1$, such a $Z$ exists density of the local structure. On the other hand if $A_0 |_e B_0$ and $A_1 |_e B_1$, and there is no $Z <_e A_0$ such that $Z \not \leq_e B_0$, then $(A_0,B_0)$ is an Ahmad pair and $B_0$ is $\highh$. Since $A_1 \geq_e B_0$, it must be $\highh$ as well. But then $(A_1,B_1)$ is not an Ahmad pair by Corollary~\ref{cor:pairiff}, and so there is a $Z <_e A_1$ with $Z \not \leq_e B_1$. 
\item If $B_0$ is an essential right half of an Ahmad $(n+1)$-pair (witnessed by $\hat A$), then it is $\highh$ by Theorem~\ref{thm:righthigh2} and so cannot be the left half of any Ahmad $(m+1)$-pair by Theorem~\ref{thm:npairchar}.
\item If $(A,B)$ is an Ahmad pair in the cone above $D_0$, then by Lemma~\ref{lem:pairCone}, $A$ is $\lowww$ and $B$ is $\highh$. Therefore $(B,C)$ cannot be an Ahmad pair in the cone above $D_1$, since $B$ is not $\lowww$.\qedhere
\end{enumerate}
\end{proof}
While the fact that there are no maximal right halves is an easy observation, the lack of minimal right halves is a new result. However, this does not lead to any new consequences for the $\forall \exists$-theory, as it is $\Pi_3$ to state the existence of a minimal right half.

\section*{Acknowledgments}
I would like to thank Mariya Soskova for suggesting this problem, for providing numerous valuable references, especially access to her notes \cite{soskovaNotes} and for her help in proofreading parts of the manuscript. I am deeply grateful to my advisor, Joseph Miller, whose guidance, encouragement, and support were indispensable to the completion of this project. His meticulous proofreading of the entire paper led to several revisions, significantly improving the clarity and the quality of the work.

%
%

\bibliographystyle{plain}
\bibliography{ref} 

\end{document}